\documentclass[a4paper,12pt]{amsart}
\usepackage{amsmath,amsthm,amsfonts,amssymb,stmaryrd,mathrsfs,enumitem}
\usepackage{latexsym}
\usepackage{fullpage}
\usepackage{a4,color,palatino,fancyhdr}
\usepackage{graphicx}
\usepackage{float}  
\usepackage[small, bf, margin=90pt, tableposition=bottom]{caption}
\RequirePackage{amssymb}
\RequirePackage[T1]{fontenc}
\usepackage{amscd}
\usepackage{xypic}
\input{xy}
\xyoption{all}
\xyoption{poly}
\usepackage[all]{xy}
\usepackage{tikz}
\usepackage{cases}
\usepackage{verbatim}
\newtheorem{thm}{Theorem}[section]
 \newtheorem{cor}[thm]{Corollary}
 \newtheorem{lem}{Lemma}[section]
 \newtheorem{prop}{Proposition}[section]
 \theoremstyle{definition}
 \newtheorem{defn}{Definition}[section]
 
 \theoremstyle{remark}
 \newtheorem{rem}{Remark}
 \newtheorem{example}{Example}[section]
 
 \numberwithin{equation}{section}
 
 \DeclareMathOperator{\IM}{Im}

\begin{document}

\newcommand{\auths}[1]{\textrm{#1},}
\newcommand{\artTitle}[1]{\textsl{#1},}
\newcommand{\jTitle}[1]{\textrm{#1}}
\newcommand{\Vol}[1]{\textbf{#1}}
\newcommand{\Year}[1]{\textrm{(#1)}}
\newcommand{\Pages}[1]{\textrm{#1}}
\newcommand{\RNum}[1]{\uppercase\expandafter{\romannumeral #1\relax}}

\title[Cohomology ring of manifold arrangements]
 {Cohomology ring of manifold arrangements}

\email{}

\author{Junda Chen, Zhi L\"u and Jie Wu}
\address{School of Mathematical Sciences, East China Normal University, Shanghai,   China}
\email{jdchen@math.ecnu.edu.cn}
\address{School of Mathematical Sciences, Fudan University\\ Shanghai\\ 200433\\ China}
\email{zlu@fudan.edu.cn}
\address{School of Mathematical Sciences, Hebei Normal University\\
Shijiazhuang, China}
\email{wujie@hebtu.edu.cn}
\urladdr{http://www.math.nus.edu.sg/\~{}matwujie}
\thanks{The second author is partially supported by the grant from NSFC (No. 11971112).
}
\begin{abstract}
We study the cohomology ring of the complement $\mathcal{M}(\mathcal{A})$ of a manifold arrangement $\mathcal{A}$ in a smooth manifold $M$ without boundary. 	
 We first give the concept of monoidal presheaf on a locally geometric poset $\mathfrak{L}$, and then define the generalized Orlik--Solomon
algebra $A^*(\mathfrak{L}, \mathcal{C})$ over a commutative ring with unit, which is built by the classical Orlik--Solomon
algebra and a monoidal presheaf $\mathcal{C}$ as coefficients. Furthermore, we construct
a  monoidal presheaf $\hat{\mathcal{C}}(\mathcal{A})$ associated with $\mathcal{A}$, so that the generalized Orlik--Solomon
algebra $A^*(\mathfrak{L}, \hat{\mathcal{C}}(\mathcal{A}))$ becomes a double complex with suitable multiplication  structure and  the associated total complex $Tot(A^*(\mathfrak{L}, \hat{\mathcal{C}}(\mathcal{A})))$ is a differential algebra. Our main result is that $H^*(Tot(A^*(\mathfrak{L}, \hat{\mathcal{C}}(\mathcal{A}))))$ is isomorphic to $H^*(\mathcal{M}(\mathcal{A}))$ as algebras. Our argument is of topological with the use of a spectral sequence induced by a geometric filtration associated with $\mathcal{A}$. In particular, we also discuss the mixed Hodge complex structure on our model if $M$ and all elements in $\mathcal{A}$ are complex smooth varieties, and show that it induces the canonical mixed Hodge structure of $\mathcal{M}(\mathcal{A})$. As an application, we calculate the cohomology of chromatic configuration spaces, which agrees with many known results in some special cases. In addition,  some explicit formulas with respect to Poincar\'e polynomial and chromatic polynomial are also given.
\end{abstract}

\maketitle
\section{Introduction}
P. Orlik and L. Solomon \cite{Orlik1980} introduced the OS-algebra $A^*(\mathcal{A})$ of central hyperplane arrangements $\mathcal{A}$ over complex number field $\mathbb{C}$, and proved that $A^*(\mathcal{A})$ is isomorphic to the cohomology ring of the complement of $\mathcal{A}$. Essentially, $A^*(\mathcal{A})$ only depends on the intersection lattice $L$ of $\mathcal{A}$,  so the OS-algebra  may be defined for every geometric lattice $L$, also denoted by $A^*(L)$.  For more general case of subspace arrangement, the cohomology ring of complement is studied by de Longueville and Schultz \cite{DeLongueville2001} using topological methods and by Deligne, Goresky, and MacPherson \cite{Deligne2000} with a sheaf-theoretic approach, also see \cite{Feichtner2000,DeConcini1995,Yuzvinsky2002} for more work about subspace arrangement.

\vskip .2cm

Dupont \cite{Dupont2013} studied the cohomology of complement of algebraic hypersurface arrangements and built an OS-model  for hypersurface arrangements.
 Bibby \cite{Bibby2013} studied the abelian arrangements,  and also gave  a kind of OS-model by Leray spectral sequence. Their works lead many applications. For example, F. Callegaro \cite{Callegaro2017} and Pagaria \cite{Pagaria2017}  studied the cohomology of complements of a toric arrangements with integer coefficients, and B. Berceanu et al. \cite{Berceanu2017} studied the cohomology of partial configuration spaces of Riemann surfaces. Also see \cite{Callegaro2018, Pagaria2018} for more recent work along this direction.

\vskip .2cm
In this paper we are concerned with manifold arrangements,  introduced in \cite{Ehrenborg2014}, i.e., arrangements of smooth submanifolds (may have codimension greater than one and need not to be algebraic) with "clean" intersection. The notion of manifold arrangements is a generalization for some classical
arrangements, such as hyperplane (or hypersurface) arrangements and the configuration spaces of manifolds.
We study the cohomology ring of  complement of manifold arrangements, and give a model with "monoidal presheaf" as coefficients of manifold arrangements by an elementary and pure topological approach. In particular, we discuss the mixed Hodge structure on our model if every manifold  is a complex smooth variety.

\vskip .2cm

Let $\mathcal{A}=\{N_i\}$ be a manifold arrangement with locally geometric quasi-intersection poset $\mathfrak{L}$ in a smooth manifold $M$ without boundary, where each $N_i$ is a smooth  submanifold and a closed subset in $M$. Let $\mathcal{M}(\mathcal{A})$ be the complement of $\mathcal{A}$ in $M$, see Section \ref{notion} for precise definitions. Then  every interval $[0,p]$ as a lattice defines an OS-algebra $A^*([0,p])$ in the sense of  Orlik and  Solomon, which encodes the  local 
combinatorial data of $\mathfrak{L}$ and plays an  important role in our discussion.
Furthermore, we will carry out our work over a commutative ring $R$ with unit as follows:
\begin{itemize}
\item[(A)] First we  generalize the OS-algebra for lattices into a "global" version $A^*(\mathfrak{L},\mathcal{C})$, which is built by $A^*([0,p])$ and a "presheaf" $\mathcal{C}$ as coefficients, where the presheaf on a poset is a concept defined in Section \ref{presheaf}. For any manifold arrangement $\mathcal{A}$, there is an associated presheaf $\mathcal{C}(\mathcal{A})$ as a cochain complex,  as we will see in Section \ref{manifold arrangement}, 
    which encodes the topological data of $\mathcal{A}$. Moreover, two differentials on $A^*([0,p])$ and $\mathcal{C}(\mathcal{A})$ induce a double complex structure on $A^*(\mathfrak{L},\mathcal{C}(\mathcal{A}))$ with differentials $\partial$ and $\delta$. In particular, this double complex also becomes a total complex  with differential $\partial+\delta$ of degree 1, denoted by $Tot(A^*(\mathfrak{L},\mathcal{C}(\mathcal{A})))$, which can be regarded as a cochain complex.  We first show that $H^*(Tot(A^*(\mathfrak{L},\mathcal{C}(\mathcal{A}))))$ is isomorphic to $H^*(\mathcal{M}(\mathcal{A}))$ as modules.
    This result can be understood as a "categorification" of M\"{o}bius inversion formula, which is essentially based upon a filtration of $A^*(\mathfrak{L},\mathcal{C}(\mathcal{A}))$ induced by a geometric filtration associated with $\mathcal{A}$. In particular, if $N_i$ and $M$ are all complex smooth varieties, then we can  modify  $\mathcal{C}(\mathcal{A})$ into $\mathcal{K}(\mathcal{A})$, such that the model $Tot(A^*(\mathfrak{L},\mathcal{K}(\mathcal{A}))$ becomes a mixed Hodge complex, inducing the canonical mixed Hodge structure of $H^*(\mathcal{M}(\mathcal{A}))$.
\item[(B)] Generally, $A^*(\mathfrak{L},\mathcal{C})$ is not a differential algebra unless the presheaf $\mathcal{C}$ has a suitable product structure (in this case, we call it a "monoidal presheaf" in Section \ref{sec mono}). Indeed, the cochain complex $\mathcal{C}(\mathcal{A})$ is not a monoidal presheaf under the natural cup product in general. We construct a  monoidal presheaf $\hat{\mathcal{C}}(\mathcal{A})$ on $\mathfrak{L}$ as a 'flat' version of $\mathcal{C}(\mathcal{A})$ such that their cohomology groups $H^*(\hat{\mathcal{C}}(\mathcal{A}))$ and $H^*(\mathcal{C}(\mathcal{A}))$ are isomorphic. Then we obtain a generalized OS-algebra $A^*(\mathfrak{L},\hat{\mathcal{C}}(\mathcal{A}))$, abbreviate it as $A^*(\mathcal{A})$. We will prove that $A^*(\mathcal{A})$ is a double complex with suitable multiplication  structure such that the associated total complex $Tot(A^*(\mathcal{A}))$ is a differential algebra, whose cohomology can be realized as $H^*(\mathcal{M}(\mathcal{A}))$ as algebras.
\end{itemize}

Our main result is stated as follows:
\begin{thm} As algebras,
$$H^*(Tot(A^*(\mathcal{A})))\cong H^*(\mathcal{M}(\mathcal{A})).$$
\end{thm}
Furthermore, the spectral sequence associated with double complex $A^*(\mathcal{A})$ is a spectral sequence of algebra.
\begin{cor}\label{main spec}
	There is a spectral sequence of algebra with  $E_1^{*,*}\cong A^*(\mathfrak{L},H^*(\mathcal{C}(\mathcal{A})))$ converges to $H^*(\mathcal{M}(\mathcal{A}))$ as algebras.
\end{cor}
From now on, we always use $H^*(\mathcal{A})$ to denote $H^*(\mathcal{C}(\mathcal{A}))$ for convenience.

\vskip .2cm

If $M$ and every $N_i\in \mathcal{C}$ are complex smooth varieties, we do a little modification on $\mathcal{C}(\mathcal{A})$ to a new complex $\mathcal{K}(\mathcal{A})$, such that $\mathcal{K}(\mathcal{A})$ is  a presheaf of mixed Hodge complex and has the same cohomology as $H^*(\mathcal{A})$. Then we have
\begin{thm}\label{mix hodge}
	The model $Tot(A^*(\mathfrak{L},\mathcal{K}(\mathcal{A})))$ is a mixed Hodge complex and
	$$H^*(Tot(A^*(\mathfrak{L},\mathcal{K}(\mathcal{A}))))\cong H^*(\mathcal{M}(\mathcal{A}))$$ with mixed Hodge structure.
\end{thm}
\begin{cor} 
If  $M$ is projective and using $\mathbb{Q}$ coefficients, then
	$$Gr_n^WH^{n-i}(\mathcal{M}(\mathcal{A}))\cong H^{-i}(A^*(\mathfrak{L},H^{n}(\mathcal{A})),\partial)$$ and $$H^*(\mathcal{M}(\mathcal{A}))  \cong H^{-*}(A^*(\mathfrak{L},H^{*}(\mathcal{A})),\partial)$$ as algebras (probably with different mixed Hodge structures).
\end{cor}

As an application, we further study the cohomology ring of the chromatic configuration space defined as $$ F(M,G) = \{(x_1,...,x_n)\in M^n|(i,j)\in E(G)\Rightarrow x_i\neq x_j\}$$
where $M$ is a smooth manifold without boundary, $G$ is a simple graph on vertex set $[n]=\{1, 2, ..., n\}$ and $E(G)$ is the edge set of $G$. Clearly $F(M,G)$ would be the classical configuration space $F(M,n)$ while $G$ is a complete graph. This concept of $F(M,G)$ was  first introduced by Eastwood and Huggett \cite{Eastwood2007} and many authors used different names, see \cite{Baranovsky2012, Berceanu2017, Wiltshire-gordon, Li2019}, where we prefer to call it the "chromatic configuration space"
 as used in \cite{Li2019} since it is the space of all vertex-colorings of $G$ such that adjacent vertices receive different colors  if we consider $M$ as a space of colors. This generalizes the concepts of graphic arrangements and configuration spaces,  and it is also an example of  manifold arrangements.
 \vskip .2cm
Historically,  F. Cohen \cite{Cohen1995} determined the cohomology ring $H^*(F(\mathbb{R}^n,q))$ for $n\geq 2$. Fulton and MacPherson \cite{Fulton1994} gave a compactification of the configuration space of a algebraic variety   and use that to compute $H^*(F(M,n))$. Kriz \cite{Kriz1994} and Totaro \cite{Totaro1996}  improved Fulton and MacPherson's result by different methods, showing that if $M$ is a smooth complex projective variety,  the rational cohomology ring of $F(M,n)$ is determined by the rational cohomology ring of $M$. The homology group of chromatic space $F(M,G)$ has been studied by Baranovsky and Sazdanovi\'{c} \cite{Baranovsky2012}. The compact support cohomology of $F(M,G)$ was studied by Petersen \cite{Petersen2018}. These works in \cite{Baranovsky2012} and \cite{Petersen2018} actually coincide from the viewpoint of  Poincar\'{e} duality. On the singular cohomology ring of chromatic configuration spaces, it seems that the only work is what Berceanu \cite{Berceanu2017} does do recently, where
Berceanu studied $H^*(F(M,G))$ by Dupont's result \cite{Dupont2013} when $M$ is a Riemann surface. If $M$ is a manifold without boundary, we show in Section \ref{chromatic space} that $F(M,G)$ is the complement of a manifold arrangement $\mathcal{A}_G$.  Then  using the approach developed in this paper, we obtain an immediate corollary about cohomology ring $H^*(F(M,G))$.

\begin{cor}
	$H^*(Tot(A^*(\mathcal{A}_G)))$ is isomorphic to $H^*(F(M,G))$ as algebras, where  $\mathcal{A}_G$ is the associated manifold arrangement such that $\mathcal{M}(\mathcal{A}_G)=F(M,G)$.
\end{cor}

In some special cases, we can express $H^*(Tot(A^*(\mathcal{A}_G)))$ more explicitly .

\begin{cor}\label{mhs}
	Assume that $M$ is a complex projective smooth variety and we use $\mathbb{Q}$ as coefficients. Then
	$$Gr_n^WH^k(F(M,G)) \cong H^{-(n-k)}(A^*(L_G,H^{n}(\mathcal{A}_G)),\partial)$$
	and $$H^*(F(M,G)) \cong H^{-*}(A^*(L_G,H^{n}(\mathcal{A}_G)),\partial)$$ as algebras.
\end{cor}

\begin{cor}\label{ass}
	Assume that the diagonal cohomology class\footnote{See Section \ref{explicit result}} of $M^2$ is zero and we use
	$\mathbb{Z}_2$ as coefficients ($M$ may not be a variety). Then $A^*(L_G,H^*(\mathcal{A}_G))$ is completely determined by $G$ and $H^*(M)$, and there exists a filtration of $H^*(F(M,G))$ such that $Gr(H^*(F(M,G)))$ is isomorphic to $A^*(L_G,H^*(\mathcal{A}_G))$ as algebras. See Theorem \ref{str of presheaf} for an explicit algebra structure of $A^*(L_G,H^*(\mathcal{A}_G))$
\end{cor}

By calculating the Poincar\'{e} polynomial of each algebra, we have
\begin{cor} With the same assumption as in Corollary~\ref{ass}, let  $G$ be a simple graph on vertex set $[n]$ and $M$ be a manifold of dimension $m$. Then
	$$P(F(M,G))=(-1)^nt^{n(m-1)}\chi_G(-P(M)t^{1-m})$$ where $P(-)$ denotes the $\mathbb{Z}_2$-Poincar\'{e} polynomial with variable $t$, and $\chi_G(t)$ is the chromatic polynomial of $G$.
\end{cor}

In a more special case, we can overcome the gap between $Gr(H^*(F(M,G)))$ and $H^*(F(M,G))$.

\begin{thm}\label{c}
	Using $\mathbb{Z}_2$ as coefficients, assume that $M$ is a manifold of dimension $m$ such that the diagonal cohomology class of $M^2$ is zero and $H^i(M)=0$ for all $i\geq (m-1)/2$. Then $H^*(F(M,G))$ is isomorphic to $A^*(L_G,H^*(\mathcal{A}_G))$ as algebras.
\end{thm}

\noindent \textbf{Comparation to known results.}
\begin{itemize}
	\item  If we forget the multiplicative structure of spectral sequence in Corollary \ref{main spec}, then it is equivalent to that constructed by Tosteson \cite{Tosteson2016}. We write it separately in Corollary \ref{immed cor} because the proof is elementary and can be extended to prove our main Theorem \ref{main alg} with a multiplicative structure involved. In \cite{Petersen2017}, Petersen also gives a spectral sequence that converges to 
	compactly supported cohomology of top stratum of a stratified space, related to Tosteson's result by Verdier duality.
	
	\item Theorem~\ref{c} agrees with the classical result of $H^*(F(\mathbb{R}^m,n))$.
	\item Using the proof process of Theorem~\ref{c}, we may reprove Orlik-Solomon's result for the complement of complex hyperplane arrangements.
\item We may reprove the Zaslavsky's result \cite{Zaslavsky1975} for the $f$-polynomial of hyperplane arrangements in $\mathbb{R}^d$ (see Corollary \ref{f-poly}).
	\item Theorem \ref{main alg} largely generalizes Dupont's result \cite{Dupont2013} to general submanifolds without the hypothesis of codimension one and being algebraic. Furthermore, Theorem \ref{ring}
 coincides with Dupont's result if we restrict to the case of  algebraic hypersurface arrangements.
	\item Let $G$ be a complete graph on $[n]$. Then Corollary~\ref{explicit chro alg} also agrees with the Kriz \cite{Kriz1994} and Totaro \cite{Totaro1996}'s model (denoted it by $E(n)$) of $F(M,n)$ for smooth complex projective variety $M$.
	\item We may reprove Eastwood and Huggett's result \cite{Eastwood2007} for  Euler characteristic of $F(M,G)$ by the spectral sequence in Theorem~\ref{SS of chromatic}.  Corollary 1.5 extends this result to the Poincar\'{e} polynomial in a special case.
	\item Baranovsky and Sazdanovi\'{c} \cite{Baranovsky2012} defined a complex, denoted it by $\mathcal{C}_{BS}(G)$, as $E_1$ page of a spectral sequence converge to  $H_*(F(M,G))$. Recently, M. B\"{o}kstedt and E. Minuz \cite{Bokstedt2019} studied the relation between $\mathcal{C}_{BS}(G)$ and Kriz-Totaro's model $E(n)$,  defined a dual $\mathcal{C}_{BS}(G)^*$ and generalized $E(n)$ to a model $R_n(A,G)$. They  proved that $\mathcal{C}_{BS}(G)^*$ and $R_n(A,G)$ are quasi-isomorphic, and $\mathcal{C}_{BS}(G)^*$ is $E_1$ page of a spectral sequence converge to $H^*(F(M,G))$ as modules,  giving a question {\em whether the spectral sequence has a product structure}. Our work affirmatively answers this question. As discussed in Section 5, our spectral sequence is equipped with a product structure and the $E_1$ page of  spectral sequence in Theorem~\ref{str of presheaf} coincides with $R_n(A,G)$ as algebras.  The product structure of our spectral sequence seems to be new.
\end{itemize}

This paper is organized as follows.    In Section 2, we review the concepts of manifold arrangements and geometric lattice, and introduce our concept of "momoidal presheaf". In Section 3, we introduce some filtration induced by manifold arrangements, and use it to prove our main theorem of module structure and discuss the mixed Hodge complex structure of our model in Section 4. In Section 5, we discuss the product structure on our construction, and then prove an algebra version of our main theorem. We introduce the chromatic configuration space in Section 6.  Then, as an application of main result, we give some explicit formulas in some special cases. In appendix, we review some results for the spectral sequence of filtrated differential algebra, and then give the proof of Theorem \ref{str of presheaf}.

\vskip .2cm
 Throughout Sections 2--5 of this paper, we always choose  a commutative ring $R$ with unit  as default coefficients.

\section{Preliminaries}\label{notion}

\subsection{Geometric lattice}
Here we are only concerned with finite posets.
A lattice is a poset $L$ in which any two elements have both a supremum (join $\vee$), and an infimum (meet $\wedge$).
Denote the unique minimal element (resp. maximal element) by $\mathbf{0}$ (resp. $\mathbf{1}$).
A lattice is {\em ranked} if all maximal chains between two
elements have the same length. By
$r(p)$ we denote the length between $\mathbf{0}$ and $p$, which is called  the rank of $p$. Of course, $r(\mathbf{0})=0$. An element of rank $1$ is called an {\em atom}. If $p>q$ and $r(p)=r(q)+1$, then we say that $p$ {\em covers} $q$, written as $p:>q$ or $q<:p$.
A {\em geometric lattice} (or {\em matroid lattice}) is a ranked lattice subject to
$$r(p\wedge q)+r(p\vee q) \leq r(p)+ r(q)$$ and every element in $L-\{\mathbf{0}\}$ is the
join of some set of atoms. For more details about lattices and geometric lattices, see \cite{Gratzer1996}.

\begin{example}\label{ex graph}
	Let $G$ denote a undirected simple graph with vertex set $[n]=\{1, ..., n\}$.
A spanning subgraph of $G$ is the subgraph such that its vertex set is $[n]$ and its edge set is a subset of all edges of $G$. We use $L_G$ to denote the lattice consisting of all partitions of vertices  induced by connected components of spanning subgraphs of $G$, so $\mathbf{0}\in L_G$ is the partition containing every single vertex and $\mathbf{1}\in L_G$ is the partition induced by connected components of $G$. It is well-known that this $L_G$ is a geometric lattice, also called the bond-lattice by Rota \cite{Rota1964}. For any partition $p\in L_G$, by $|p|$ we denote the length of $p$ as a partition, so $r(p)=n-|p|$.  There is a well-known relation between the M\"{o}bius function $\mu$ of $L_G$ and the chromatic polynomial $\chi_G(t)$ of graph $G$ as follows: $$\chi_G(t)=\sum_{p\in L_G} \mu(0,p)t^{n-r(p)}$$
where the M\"{o}bius function $\mu$ is an integer valued function on $L_G\times L_G$ defined recursively by $\mu(p, p)=1$, $\sum_{p\leq l\leq q}\mu(p,l)=0$ if $p<q$ and $\mu(p,q)=0$ otherwise.
 Given a partition $p\in L_G$, we can remove those edges in $G$ of  crossing different components of $p$,  and then get a subgraph of $G$,  denoted by $G|p$. It is easy to check that  the associated lattice $L_{G|p}$ of this subgraph is the interval $[0,p]\subset L_G$. This fact will be used in Section \ref{chromatic space}.
\end{example}

The following property is important for geometric lattice, see \cite[Chapter IV]{Gratzer1996}.
\begin{prop}\label{geo-lattice}
	Let $L$ be a geometric lattice. Then every interval $[a,b]$ of $L$ is also a geometric lattice.
\end{prop}

For  some elements $p_i$ of $L$,  $r(\vee_i p_i)\leq \sum_i r(p_i)$ by definition of geometric lattice. We say that the $p_i$'s are {\em independent} if $r(\vee_i p_i)= \sum_i r(p_i)$ and {\em dependent} otherwise. Here is an easy lemma we will use later.
\begin{lem}
	(i) If $a$ is an atom and $p\in L$ with $a\nleq p$ and $p\neq \mathbf{0}$, then $a,p$ are independent, i.e.,  $r(p\vee a)=r(p)+1$.
	(ii) Every element $p\neq \mathbf{0}$ is the join of some independent atoms $a_i$, $1\leq i\leq r(p)$. (iii) For dependent $p,q$ and write $q=\vee a_i$ as join of independent atoms, there exists an integer $s$ such that $a_s \leq p\vee a_1\vee a_2\vee\cdots \vee a_{s-1}$.
\end{lem}
\begin{proof}
	(i) First $r(a\wedge p)+r(a\vee p) \leq r(a)+ r(p)$ by definition.  Since $a$ is an atom and $a\nleq p$, we have that  $a\wedge p=\mathbf{0}$ so $r(a\vee p) \leq r(p)+1$, and  $a\vee p>p$ so $r(a\vee p)>r(p)$. Thus, $r(p\vee a)=r(p)+1$.
	(ii) Clearly $p$ is always the join of some atoms $\{a_i\}$ by definition. Remove the elements $a_s$ with $a_s\leq a_1\vee a_2\vee\cdots \vee a_{s-1}$, we get the required independent set of atoms by (i). (iii) If not, $r(p\vee a_1\vee a_2\vee\cdots \vee a_{s})=r(p)+s$ by induction. Furthermore,  $r(p\vee q)=r(p)+r(q)$, which contradicts to the condition that $p,q$ are dependent.
\end{proof}
\subsection{Locally Geometric lattice}

\begin{defn}\label{locally geometric}
	A poset $\mathfrak{L}$ is {\em locally geometric} if it has a minimal element  $\mathbf{0}$ and the interval $[\mathbf{0},p]$ is a geometric lattice for any $p\in \mathfrak{L}$.
\end{defn}
\begin{rem}
	It is known that the intersection lattice of hyperplane arrangements is a geometric lattice. For more general case,
	the poset of layers (connected components of intersections) of hypersurface arrangements are always locally geometric since these arrangements locally "looks like" hyperplane arrangements. The above definition is a combinatorial abstraction of poset of layers.
\end{rem}
	From now on, we always denote geometric lattice (resp. locally geometric poset) as $L$ (resp. $\mathfrak{L}$) and write $[p,\infty)=\{q\in \mathfrak{L}|q\geq p\}$. The poset $[p,\infty)$ is also locally geometric since interval $[p,q]$ of geometric lattice $[\mathbf{0},p]$ is also geometric lattice.	
\vskip .2cm

There are some easy properties of locally geometric poset.
\begin{prop}\label{canonical map}
	Let $\mathfrak{L}$ be a locally geometric poset.
\begin{enumerate}
	\item Set $r(p)=r([\mathbf{0},p])$ for $p\in \mathfrak{L}$. Then $r(p)$ is a grading of poset $\mathfrak{L}$, where $r([\mathbf{0},p])$ in the right side is the rank of geometric lattice $[\mathbf{0},p]$.
	\item In general, two elements $p,q \in  \mathfrak{L}$ may have more than one minimal upper bounds, we denote this set by $p \mathring{\vee} q$ as "global" join of $p,q$. Notice that for any $s\in p \mathring{\vee} q$, $s=p\vee q$ in lattice $[\mathbf{0},s]$.
	\item Given $p_1,q_1,p_2,q_2\in \mathfrak{L}$ such that $p_2\leq p_1, q_2\leq q_1$, then for every $s\in p_1 \mathring{\vee} q_1$, there is one and only one element $t\in p_2 \mathring{\vee} q_2$ such that $t\leq s$. This property gives us a canonical map $\lambda_{p_1q_1p_2q_2}:p_1 \mathring{\vee} q_1 \rightarrow p_2 \mathring{\vee} q_2$. 	
	
\end{enumerate}
\end{prop}
\begin{proof}
(1) and (2) are obvious. It suffices to  prove (3). For any $s\in p_1 \mathring{\vee} q_1$, choose an element $t=p_2\vee q_2$ in lattice $[\mathbf{0},s]$, we claim that it is a minimal upper bound of $p_2,q_2$ in $\mathfrak{L}$. Otherwise, we will have another upper bound $t_1\le t$, a contradiction to $t=p_2\vee q_2$ in lattice $[\mathbf{0},s]$. For the "only one" part, assume that $t_1,t_2\in p_2 \mathring{\vee} q_2$ with $t_1,t_2\leq s$ and $t_1\neq t_2$. It is easy to see that $t_1,t_2$ are also minimal upper bound of  $p_2,q_2$ in $[\mathbf{0},s]$, but this is impossible since $[\mathbf{0},s]$ is a lattice.
\end{proof}

\subsection{Orlik-Solomon algebra}
P. Orlik and L. Solomon introduced the OS-algebra $A^*(L)$ in \cite{Orlik1980} for any finite geometric lattice  $L$.

\vskip .2cm

Let  ${\rm Atom}(L)$ be the set of all atoms of $L$. Recall a subset $S\subset {\rm Atom}(L)$ is said to be dependent if $r(\vee S)<|S|$.
Let $E^*(L)$ be the exterior algebra over a commutative ring $R$ with unit 1, generated by the
elements $e_a, a\in {\rm Atom}(L)$  with the basis
$e_S=e_{a_1}\cdots e_{a_k}$,
$S=\{a_1,.., a_k\}\subset {\rm Atom}(L)$, where $e_S=1$ if $S$ is empty.  $E^*(L)$ admits the natural derivation
$\partial: E^*(L)\longrightarrow E^*(L)$ given by
$$\partial e_S=
\begin{cases}
0 & \text{ if $S$ is empty}\\
1 & \text{ if } S=\{a\}\\
\sum_{j=1}^k(-1)^{j-1}e_{a_1}\cdots \widehat{e_{a_j}}\cdots e_{a_k} & \text{ if } S=\{a_1,.., a_k\}.
\end{cases}$$
Let $\mathcal{I}(L)$  be the ideal in $E^*(L)$,
generated by $\partial e_S$ for all dependent sets $S\subset {\rm Atom}(L)$.
Then the Orlik-Solomon algebra of  $L$
is defined as  the graded commutative quotient $R$-algebra
$A^*(L)=E^*(L)/\mathcal{I}(L)$.

\vskip .2cm

We list some well known properties of the OS-algebra here. For more details, see \cite{Yuzvinsky2001, Dimca2009}.

\begin{prop}~ \label{OS-property}
	\begin{enumerate}
		\item
	The OS algebra $A^*(L)$ is a $L$-graded algebra, i.e., $A^*(L)=\bigoplus_{p\in L}A^*(L)_p$, where $A^*(L)_p$ denotes the homogeneous submodule of order $p$, which is also a free $R$-module.
	\item
	$A^*(L)_p\cdot A^*(L)_q \subseteq A^*(L)_{p\vee q}$. If $r(p\vee q)<r(p)+r(q)$, then $A^*(L)_p\cdot A^*(L)_q=0$.
	\item
	$A^*([\mathbf{0},p])$ is naturally isomorphic to the sub-algebra $\bigoplus_{q\leq p}A^*([\mathbf{0},s])_q \subset A^*([\mathbf{0},s])$ for every $p \leq s$. Denote the imbedding map $A^*([\mathbf{0},p])\hookrightarrow A^*([\mathbf{0},s])$ by $i_{s}$.
	\item
	The derivation $\partial$ on $E^*(L)$ induces a derivation (still denoted by $\partial$) on the OS-algebra $A^*(L)$, which maps $A^*(L)_p$ to $\bigoplus_{p:>q}A^*(L)_q$.
In particular, $(A^*([\mathbf{0},p]),\partial)$ is an exact complex for every $p\neq \mathbf{0}, p\in L$.
	\item
	$\dim A^*(L)_p=(-1)^{r(p)}\mu(0,p)$ for any field as coefficients, where $\mu(-,-)$ is the M\"{o}bius function of $L$.
\end{enumerate}
\end{prop}

	Since we mainly consider the case of locally geometric posets in this paper, naturally we wish to know {\em whether  OS-algebra can still be defined on locally geometric posets or not.}  We will answer this in Sections \ref{presheaf}--\ref{monoidal presheaf}.

\subsection{Manifold arrangements}\label{manifold arrangements}
Given a connected manifold $M$  and a finite collection of  submanifolds $\mathcal{A}=\{N_i\}$, where $M$ and each $N_i$ are smooth without boundaries. As defined in \cite{Ehrenborg2014},   $\mathcal{A}$ is said to be a {\em manifold arrangement} if it satisfies the Bott's clean intersection property that for every $x\in M$, there exist
a neighborhood $U$ of $x$, a neighborhood $W$ of the origin in $ \mathbb{R}^n$, a subspace arrangement $\{V_i\}$ in $ \mathbb{R}^n$
and a diffeomorphism $\phi : U\rightarrow W$  such that $\phi$ maps $x$ to the origin and maps $\{N_i\cap U\}$ to $\{V_i\cap W\}$.
\vskip .2cm

 Roughly speaking, a manifold arrangement is 'locally diffeomorphic' to a subspace arrangement in an Euclidean space. There are some different combinatorical structures associated with manifold arrangements $\mathcal{A}$ (e.g. the poset of all possible intersection of each $\mathcal{A}$). In the case of hypersurface arrangements, many authors prefer the poset of layers. R. Ehrenborg and M. Readdy \cite{Ehrenborg2014} introduced the concept of \emph{intersection poset} as a flexible tactic that there could be several suitable intersection posets depending on particular circumstances. Roughly speaking, they combined some layers together by a suitable way. However, some condition in their definition of intersection poset is unnecessary in this paper, we remove these condition and define  \emph{quasi-intersection poset}.
 
 \vskip .2cm

 Recall that an intersection of $\mathcal{A}$ is the intersection of a subset of $\mathcal{A}$, and we always let the intersection of empty set be $M$. A connect component of a nonempty intersection is called a layer.
 \begin{defn}
 	For any nonempty intersection with connected components $c_1, ..., c_k$, any disjoint union $c_{i_1}\sqcup c_{i_2}\sqcup \cdots\sqcup c_{i_s}$ is called a \emph{quasi-layer}. All quasi-layer form a poset $\mathfrak{P}$ ordered by reverse inclusion. For any $p\in \mathfrak{P}$, let $M_p$ be the associated quasi-layer (it is the same thing of $p$, we use the symbol $M_p$ to emphasize that it is a submanifold of $M$ rather than an element of poset).	
 	A \emph{quasi-intersection poset} $P$ of manifold arrangement $\mathcal{A}$ is a sub-poset of $\mathfrak{P}$ such that:
 	\begin{enumerate}
 		\item $M$ is the minimum element;
 		\item $M_p \cap M_q$ equals disjoint union $\sqcup_{s\in p \mathring{\vee} q}M_s$ for all $p,q\in P$.
 	\end{enumerate}

 \end{defn}

\vskip .2cm
In this paper, we always assume the following two {\em additional conditions} for manifold arrangement $\mathcal{A}$ that are enough for our application:
\begin{itemize}
	\item every submanifold $N_i$ is a closed subset of $M$;
	\item there exists a quasi-intersection poset $\mathfrak{L}$ of $\mathcal{A}$ such that $\mathfrak{L}$ is locally geometric poset.
\end{itemize}

Given $p\in \mathfrak{L}$, let $M_p$ be the associated quasi-layer. Set  $S_p=M_p-\bigcup_{q>p}M_q$ and write $\mathcal{M}(\mathcal{A})=S_{\textbf{0}}$. By $\mathcal{A}_p$ we denote the collection $\{M_q|q:>p\}$. Then  $\mathcal{A}_p$ is a manifold arrangement in $M_p$ with an intersection poset $[p,\infty)$ and $S_p= \mathcal{M}(\mathcal{A}_p)$.

\subsection{Presheaf and copresheaf on poset}\label{presheaf}
We see from Proposition~\ref{OS-property} that $(A^*([\mathbf{0},p]),\partial)$ is  an exact complex, i.e., its all homologies vanish. If we combine this complex with suitable "coefficients", then the corresponding  homologies  will be more interesting.

\vskip .2cm
For $p, q$ in a poset $P$, $p\leq q$,  we may understand that there is a unique morphism $p\rightarrow q$, so $P$ may be regarded as a category.

\begin{defn} A \textbf{presheaf} on a poset $P$ is a contravariant functor $\mathcal{C}$ from $P$  to the category of $R$-modules (or algebras), by mapping $p \longmapsto \mathcal{C}_p$ and  mapping every $q\rightarrow p$ to $f_{p,q}:\mathcal{C}_p \rightarrow \mathcal{C}_q$,  satisfying
	\begin{enumerate}
		\item
		$f_{p,p}=id$;
		\item
		$f_{q,s}\circ f_{p,q}=f_{p,s}$.
	\end{enumerate}
All presheaves  on $P$ with natural transformations as  morphisms also form a category. Similarly, we may also define a covariant functor from $P$ to the category of $R$-modules (or algebras) by mapping every $p\rightarrow q$ to $f_{p,q}:\mathcal{C}_p \rightarrow \mathcal{C}_q$, which is called a \textbf{copresheaf} on $P$.
\end{defn}
\begin{rem}
The definition of presheaves on $P$ is a special case of presheaves on category as in \cite{Artin1962}.  We follow this statement and use the terminology "presheaf (copresheaf)" as well. 
\end{rem}
\begin{example}\label{ex}
	Let $\mathcal{A}$ be a manifold arrangement in $M$ with a quasi-intersection poset $\mathfrak{L}$.
	The cochains $\mathcal{C}(\mathcal{A})_p=C^*(M,M-M_p)$ combined with inclusion maps $$f_{p,q}:C^*(M,M-M_p) \rightarrow C^*(M,M-M_q)$$ for $p\geq q$ give a presheaf on $\mathfrak{L}$. Similarly, let $H^*(\mathcal{A})_p=H^*(M,M-M_p)$, $H^*(\mathcal{A})$ is also a presheaf on $\mathfrak{L}$. We will see more details in next section.
\end{example}
\begin{example}\label{example of copresheaf}
	With the same assumption as Example~\ref{ex},
	the cohomology rings $H^*(M_p)$ combined with $\phi^*_{q,p}$ give a copresheaf (of rings) on $\mathfrak{L}$, where $\phi^*_{q,p}$ is the cohomology homomorphism induced by the inclusion map $ \phi_{q,p}: M_q\rightarrow M_p$ for $q\geq p$.
\end{example}
Now, we define a  complex $(A^*(\mathfrak{L},\mathcal{C}),\partial)$ for a presheaf $\mathcal{C}$ on a locally geometric poset $\mathfrak{L}$, which is a generalization of the chain complex structure on OS-algebra.

\begin{defn}\label{OS complex}
	Let $\mathcal{C}$ is a presheaf on $\mathfrak{L}$. Define
	$$A^*(\mathfrak{L},\mathcal{C})_p= A^*([\mathbf{0},p])_p\otimes \mathcal{C}_p$$
		$$A^*(\mathfrak{L},\mathcal{C})=\bigoplus_{p\in \mathfrak{L}}A^*(\mathfrak{L},\mathcal{C})_p$$
with the differential $\partial$ defined by
	$$\partial (x\otimes c) = \sum_i x_i\otimes f_{p,p_i}(c)$$
	 for $x\in A^*([\mathbf{0},p])_p$, $c \in \mathcal{C}_p $ and $\partial x = \sum_i x_i$ in $A^*([\mathbf{0},p])$ such that $x_i \in A^*([\mathbf{0},p])_{p_i}$ and $p$ covers $p_i$.
\end{defn}

It needs to check that $(A^*(\mathfrak{L},\mathcal{C}),\partial)$ is a well-defined chain complex. Actually,  let $\partial x_i=\sum_j x_{i,j}$ such that $x_{i,j}$ belongs to some $A^*([\mathbf{0},p])_{p_j}$ where $r(p_j)=r(p_i)-1=r(p)-2$. Then we see that $\sum_{i}x_{i,j}=0$ for every $j$ since $\partial\partial(x)=0$. Thus $\partial\partial(x\otimes c)=\sum_{i,j}x_{i,j}\otimes f_{p,p_{j}}(c)=\sum_j(\sum_ix_{i,j})\otimes f_{p,p_{j}}(c)=0$, as desired.

\vskip .2cm
 Put  a {\em negative} grading on $A^*(\mathfrak{L},\mathcal{C})$ such that $A^*(\mathfrak{L},\mathcal{C})_{-i} = \bigoplus_{r(p)=i}A^*([\mathbf{0},p])_p\otimes \mathcal{C}_p$. The negative grading will be convenient for the construction of double complex and the use of spectral sequence later.
For any subset $[p,\infty)\subset \mathfrak{L}$, we will use $A^*([p,\infty),\mathcal{C})$ to denote $A^*([p,\infty),\mathcal{C}|_{[p,\infty)})$ for a convenience, where $\mathcal{C}|_{[p,\infty)}$ means the restriction of  $\mathcal{C}$ on poset $[p,\infty)$. Then we can check easily that all $\mathcal{C} \longmapsto A^*([p,\infty),\mathcal{C})$ define a functor from presheaves  on $\mathfrak{L}$ to the category of $R$-module complexes.

\begin{rem}
	Actually we may also define a functor $$\Gamma_{p}(\mathcal{C})=\text{Coker}(\bigoplus_{\alpha>p}\mathcal{C}_\alpha \xrightarrow{\oplus f_{\alpha,p}} \mathcal{C}_p)$$
It can easily be checked  that $\Gamma_{p}$ is right exact and the left derived functor $L_i\Gamma_{p}$ is naturally isomorphic to $H^{-i}(A^*([p,\infty),\mathcal{C}),\partial)$. However, we do not use this fact in this article.
\end{rem}

The following definition and  lemma will be used later.

\begin{defn}
	Let $\alpha\in \mathfrak{L}$ and $A$ be any $R$-module. Define a presheaf $j_{\alpha*}A$ on $\mathfrak{L}$ as
	\begin{subnumcases}
	{(j_{\alpha*}A)_p=}
	A, &$p\leq \alpha$ \nonumber\\
	0, &otherwise \nonumber
	\end{subnumcases}
	where the map $f_{p,q}$ is the identity if $q\leq p \leq \alpha$ and zero otherwise, which is called the \textbf{sky-scraper presheaf}.
\end{defn}

The following lemma is a direct result of the exactness of the OS-algebra.
\begin{lem}\label{sky}
	For any sky-scraper presheaf $j_{\alpha*}A$, $(A^*([p,\infty),j_{\alpha*}A),\partial)$ is exact if and only if $\alpha \neq p$. If $\alpha = p$, then
 $$H^{-i}(A^*([p,\infty),j_{\alpha*}A))=
 \begin{cases} A & \text{ if } i=0\\
 0 &  \text{ if } i\neq 0.
 \end{cases}
 $$
\end{lem}

\subsection{Monoidal presheaf on lattice}\label{sec mono} With the understanding on a negative grading on $A^*(\mathfrak{L},\mathcal{C})$,
	 the cohomology $H^{-i}(A^*([p,\infty),\mathcal{C}))$ will be an $R$-algebra if $\mathcal{C}$ has a suitable product structure on it, as we will describe below. It is exactly our main approach for the calculation of the cohomology ring of the complement of manifold arrangements.	

\vskip .2cm 
\textbf{Notation.} Let $\mathcal{C}$ be a presheaf on $P$ and $I$ is a subset of $P$, denote $\mathcal{C}_I=\bigoplus_{p\in I}\mathcal{C}_p$. Let $J\subseteq P$ and there exists map $\lambda : I\rightarrow J$ that $\forall p\in I,\lambda(p)\leq p$.
Denote $f_{I,J,\lambda}=\bigoplus_{p\in I}f_{p,\lambda p}$ . In this paper, $\lambda$ is always clear in context, so we omit it and always write $f_{I,J}$.

\begin{defn}\label{monoidal presheaf}
	Let $\mathcal{C}$ be a presheaf on a locally geometric poset $\mathfrak{L}$. Then  $\mathcal{C}$ is said to be \textbf{monoidal}  if $\bigoplus_p\mathcal{C}_p$ is an associative algebra satisfying that $\mathcal{C}_s\cdot \mathcal{C}_t \subseteq \mathcal{C}_{s\mathring{\vee} t}$ for every $s,t\in \mathfrak{L}$ and	
	$$b\cdot
	f_{p,q}(a)=f_{p\mathring{\vee} s, q\mathring{\vee} s}(b\cdot a)$$ $$f_{p,q}(a)\cdot b=f_{p\mathring{\vee} s, q\mathring{\vee} s}(a\cdot b)$$ for $q\leq p,a\in \mathcal{C}_p,b\in \mathcal{C}_s$. Notice that in above notation of $f_{p\mathring{\vee} s, q\mathring{\vee} s}$, we need a map $\lambda: p\mathring{\vee} s \rightarrow q\mathring{\vee} s$, we always let it be the canonical map $\lambda_{psqs}$ given in Proposition \ref{canonical map}(3) with no confusion.
\end{defn}

\begin{rem}
By Definition~\ref{monoidal presheaf}, a monoidal presheaf $\mathcal{C}$ on a geometric lattice $L$ is actually a monoidal functor from $L$ to the category of $R$-modules if there is a unit $1\in \mathcal{C}_\mathbf{0}$, but we will not use this general terminology for simplicity. In this case,  $L$ is regarded as a monoidal category and $\vee$ is the monoidal product. 	
\end{rem}

\begin{example}
	Let $\mathcal{A}$ be a manifold arrangement in $M$ with a quasi-intersection poset $\mathfrak{L}$. Then
	the cohomology rings $H^*(\mathcal{A})_p:=H^*(M,M-M_p)$ form a monoidal presheaf with cup product $ H^*(M,M-M_p)\otimes H^*(M,M-M_q)\xrightarrow{\cup} \bigoplus_{s\in p\mathring{\vee} q }H^*(M,M-M_{s})$. This monoidal presheaf $H^*(\mathcal{A})$ will appears many times later.
\end{example}

For a monoidal presheaf $\mathcal{C}$ on $\mathfrak{L}$, making use of the OS-algebra $A^*([\mathbf{0},p])$ and the monoidal product of $\mathcal{C}$ we can define  a  product structure on $A^*(\mathfrak{L},\mathcal{C})$ as follows.

\begin{defn}\label{OS-alg with coef}
	Assume that $\mathcal{C}$ is a monoidal presheaf on a locally geometric poset $\mathfrak{L}$. Let $j_s: \bigoplus_p \mathcal{C}_p\rightarrow \mathcal{C}_s$ be the projection on $\mathcal{C}_s$. Define the product on $A^*(\mathfrak{L},\mathcal{C})$ as $$(x\otimes c_1)\cdot (y\otimes c_2)=(-1)^{\deg(c_1)r(q)}\sum_{s\in p\mathring{\vee} q}(i_{s}x\cdot i_{s}y)\otimes j_s(c_1\cdot c_2)$$ for $x\in A^*([\mathbf{0},p])_p,y\in A^*([\mathbf{0},q])_q, c_1 \in \mathcal{C}_p, c_2 \in \mathcal{C}_q$, where the first "$\cdot$" in the right side is the product of OS-algebra $A^*([\mathbf{0},s])$ and $i_s$ is the imbedding map in Proposition \ref{OS-property}(3).
\end{defn}

\begin{rem}
	The algebra $A^*(\mathfrak{L},\mathcal{C})$ can be viewed as a "global" OS-algebra with  $\mathcal{C}$ as coefficients, and it is actually a differential graded algebra. We will prove it later.
\end{rem}

\section{A presheaf and its filtration for manifold arrangements}\label{manifold arrangement}
 Let $\mathcal{A}$ be a manifold arrangement in $M$ with  quasi-intersection poset $\mathfrak{L}$. Recall that $M_p$ is the quasi-layer associated with $p\in \mathfrak{L}$, and $S_p=M_p-\bigcup_{q>p}M_q$, as defined in last section.
\vskip.2cm
Associated with the manifold arrangement $\mathcal{A}$, there is a natural presheaf $\mathcal{C}(\mathcal{A})$ that encodes topological data of $\mathcal{A}$,
 	such that $\mathcal{C}(\mathcal{A})_p= C^*(M,M-M_p)$, and $\mathcal{C}(\mathcal{A}): p \longmapsto \mathcal{C}(\mathcal{A})_p$ is a graded presheaf on $\mathfrak{L}$ with the inclusion map $f_{p,q} : \mathcal{C}(\mathcal{A})_p \rightarrow \mathcal{C}(\mathcal{A})_q$ for $p\geq q$.

\subsection{A classical filtration}
We consider a classical filtration $F_p^*M$ of  $M$ for every  $p\in \mathfrak{L}$ and then study the $E_1$-term of associated spectral sequence.  This will be very useful in the proof of our main theorem.

\begin{defn}\label{geo filt}
	Let $p\in \mathfrak{L}$. Define an increasing filtration  $F_p^0M\subset F_p^1M \subset\cdots \subset M$ by
	\begin{subnumcases}
	{F_p^iM=}
	M- M_p, &  if \ ~$i < r(p)$ \nonumber\\
	M- \bigcup_{q\geq p,r(q)=i}M_q, & if \ ~$i \geq r(p).$ \nonumber
	\end{subnumcases}
\end{defn}

These filtrations $F_p^*M, p\in \mathfrak{L}$ induce the decreasing filtrations of the presheaf $\mathcal{C}(\mathcal{A})$, which are defined as follows.

\begin{defn}\label{presheaf of SS}
	Let $p\in \mathfrak{L}$.  Define a filtration of  $\mathcal{C}(\mathcal{A})$
 $$F^0\mathcal{C}(\mathcal{A})_p\supset F^1\mathcal{C}(\mathcal{A})_p \supset\cdots \supset F^{i}\mathcal{C}(\mathcal{A})_p\supset\cdots$$ by
	$$F^i\mathcal{C}(\mathcal{A})_p=C^*(M,F_p^iM).$$

Let $E_{r,p}^{*,*}(\mathcal{A})$ be the $E_r$-term of the spectral sequence associated with the filtration $F^*\mathcal{C}(\mathcal{A})_p$.
Clearly, the inclusion map $f_{p,q}: \mathcal{C}(\mathcal{A})_p \rightarrow \mathcal{C}(\mathcal{A})_q$ gives $F^i\mathcal{C}(\mathcal{A})_p \subset F^i\mathcal{C}(\mathcal{A})_q$ by definition, so $f_{p,q}$ induces the map between two spectral sequences, denoted by $f_{p,q,r}:E_{r,p}^{*,*}(\mathcal{A})\rightarrow E_{r,q}^{*,*}(\mathcal{A})$. Thus $E_{r,*}^{*,*}(\mathcal{A})$ is also a presheaf on $\mathfrak{L}$ for every $E_r$-page.
\end{defn}

The following lemma will be useful in the study of the presheaf on $\mathfrak{L}$ for the $E_1$-page.

\begin{lem}\label{lem of filtration}
	$F_p^iM$ is an open subset of $F_p^{i+1}M$, formed by removing some  sub-manifolds $\sqcup_{\alpha\geq p,r(\alpha)=i}S_\alpha$ as closed subsets.
\end{lem}
\begin{proof}
	If $i\geq r(p)$, then
	\begin{align*}
	F_p^iM&=M- \cup_{\alpha\geq p,r(\alpha)=i}M_\alpha\\
	&=M- \sqcup_{\alpha\geq p,r(\alpha)\geq i}S_\alpha\\
	&=M- \sqcup_{\alpha\geq p,r(\alpha)\geq i+1}S_\alpha-\sqcup_{\alpha\geq p,r(\alpha)=i}S_\alpha\\
	&=M- \cup_{\alpha\geq p,r(\alpha)= i+1}M_\alpha-\sqcup_{\alpha\geq p,r(\alpha)=i}S_\alpha\\
	&=F_p^{i+1}M-\sqcup_{\alpha\geq p,r(\alpha)=i}S_\alpha
	\end{align*}
	If $i<r(p)$ , then $\sqcup_{\alpha>p,r(\alpha)=i}S_\alpha=\emptyset$ so the above equation follows from the definition of $F_p^iM$.
	A similar calculation shows that $S_\alpha=M_\alpha \cap F_p^{i+1}M$ for $\alpha\geq p$ and $r(\alpha)=i$. Then $S_\alpha$ is a closed subset of $F_p^{i+1}M$ since $M_\alpha$ is closed in $M$, so is $\sqcup_{\alpha\geq p,r(\alpha)=i}S_\alpha$ because $\mathfrak{L}$ is finite.
\end{proof}

\subsection{The presheaf of the $E_1$-page}
We will show that the  presheaf of the $E_1$-page has a simple structure. Actually it is just a direct sum of some sky-scraper presheaves.

\vskip .2cm

In the following discussion, for  a submanifold  $S$ of some manifold $X$, by $N(S)$ we denote the tubular neighborhood of $S$ in $X$ and  $N(S)_0=N(S)-S$. If $S$ is zero-codimensional, we convention that $N(S)=S$ so $N(S)_0=\emptyset$.

\begin{rem}\label{rem of excision}
	Lemma \ref{lem of filtration} tells us that we may use the excision theorem on the couple $(F_p^{i+1}M,F_p^iM)$ where $F_p^iM=F_p^{i+1}M-\sqcup_{\alpha\geq p,r(\alpha)=i}S_\alpha$.  Consider the tubular neighborhood $N_{p,i}$ of $\sqcup_{\alpha\geq p,r(\alpha)=i}S_\alpha$ in $F_p^{i+1}M$, we may write  $N_{p,i}=\sqcup_{\alpha\geq p,r(\alpha)=i} N(S_\alpha)$ where
$N(S_\alpha)$ is the tubular neighborhood of $S_\alpha$ in $F_\alpha^{i+1}M$,  and then by  $N_{p,i,0}$ we means\\ $\sqcup_{\alpha\geq p,r(\alpha)=i} N(S_\alpha)_0$. Therefore, we have that $H^*(F_p^{i+1}M,F_p^iM)= H^*( N_{p,i},N_{p,i,0})$ by the excision theorem.
\end{rem}

\begin{thm}\label{presheaf of E1}
	The presheaf $E_{1,*}^{i,j}(\mathcal{A})$ on $\mathfrak{L}$ is the direct sum of some sky-scraper presheaves as follows: $$E_{1,*}^{i,j}(\mathcal{A}) \cong \bigoplus_{r(\alpha)=i} j_{\alpha *}H^{i+j}(N(S_\alpha),N(S_\alpha)_0)$$
\end{thm}

\begin{proof}
	
	If $p>\textbf{0}$ or $i>0$, using Lemma \ref{lem of filtration} and  Remark \ref{rem of excision}, we have
	\begin{align*}
	E_{1,p}^{i,j}&=H^{i+j}(F^iC^*(M,M- M_p)/F^{i+1}C^*(M,M-M_p))\\
	&=H^{i+j}(F_p^{i+1}M,F_p^{i}M)\\
	&=H^{i+j}(N_{p,i},N_{p,i,0}).
	\end{align*}
	   Now consider the map $f_{p,q,1}$ of $E_1$-page as mentioned in Definition \ref{presheaf of SS}. There is the following commutative diagram of spaces
	\[
	\begin{CD}
	(N_{q,i},N_{q,i,0}) @>\phi_q>> (F_q^{i+1}M,F_q^{i}M)\\
	@A\varphi AA @V\chi VV\\
	(N_{p,i},N_{p,i,0}) @>\phi_p>> (F_p^{i+1}M,F_p^{i}M)
	\end{CD}
	\]
	where $\phi_q,\phi_p$ are the inclusion maps induced by excision, and $\varphi,\chi$ are also  inclusion maps.  Then we have the following commutative diagram
	\[
	\begin{CD}
	\bigoplus_{\alpha\geq q,r(\alpha)=i}H^{i+j}(N(S_\alpha),N(S_\alpha)_0) @<\phi_q^*<\cong< E_{1,q}^{i,j}\\
	@V\varphi^* VV @A A \chi^*=f_{p,q,1} A\\
	\bigoplus_{\alpha\geq p,r(\alpha)=i}H^{i+j}(N(S_\alpha),N(S_\alpha)_0) @<\phi_p^*<\cong< E_{1,p}^{i,j}.
	\end{CD}
	\]
	Now we are going to show that $\phi_q^* \circ f_{p,q,1}\circ(\phi_p^*)^{-1}$ is an inclusion map. For this, we notice that the set $\{\alpha|\alpha\geq q,r(\alpha)=i\}$ can be divided into two disjoint parts: $\{\alpha|\alpha\geq p,r(\alpha)=i\}$ and $\{\alpha|\alpha\geq q,\alpha\ngeq p,r(\alpha)=i\}$.
 Let $\varphi_0:(N(S_{\alpha_0}),N(S_{\alpha_0})_0)\rightarrow (N_{q,i},N_{q,i,0})$ be the inclusion map for any $\alpha_0$ in the second part. We first show that $\varphi_0^*\circ \phi_q^* \circ f_{p,q,1}=0$. This follows from fact that $N(S_{\alpha_0})$ is contained in $F_p^{i}M$ because $F_p^{i}M=F_p^{i+1}M-\sqcup_{\alpha\geq p,r(\alpha)=i}S_\alpha$ and $N(S_{\alpha_0})\cap N(S_{\alpha})=\emptyset$ for any $\alpha$ lying in the first part; namely $\chi\circ\phi_q\circ \varphi_0$ maps $N(S_{\alpha_0})$ into $F_p^{i}M$.  Combining with the  above  commutative diagram, we see that $f_{p,q,1}$ must be the inclusion map under the isomorphisms $\phi_q^*$ and $\phi_p^*$,  i.e.,
	\begin{subnumcases}
		{E_{1,p}^{i,j}\cong}
		\bigoplus_{\alpha\geq p,r(\alpha)=i}H^{i+j}(N(S_\alpha),N(S_\alpha)_0), &~$i\geq r(p)$ \nonumber\\
		0 &else \nonumber
	\end{subnumcases}
	and $f_{p,q,1}$ is the inclusion map $$\bigoplus_{\alpha\geq p,r(\alpha)=i}H^{i+j}(N(S_\alpha),N(S_\alpha)_0)\hookrightarrow \bigoplus_{\alpha\geq q,r(\alpha)=i}H^{i+j}(N(S_\alpha),N(S_\alpha)_0)$$ under the sense of the above isomorphisms. Thus, $E_{1,*}^{*,*}(\mathcal{A})$ is isomorphic to the direct sum of sky-scraper presheaves as follows $$E_{1,*}^{i,j}(\mathcal{A})\cong \bigoplus_{r(\alpha)=i} j_{\alpha *}H^{i+j}(N(S_\alpha),N(S_\alpha)_0).$$
\end{proof}

\section{Construction of main model}
In Section~\ref{manifold arrangement}, for a manifold arrangement $\mathcal{A}$ with a quasi-intersection lattice $\mathfrak{L}$,
 the associated presheaf $\mathcal{C}(\mathcal{A})_p$   is chosen as a cochain complex for every $p\in \mathfrak{L}$, and the structure map $f_{p,q}$ is commutative with the differential $\delta$ of every complex $\mathcal{C}(\mathcal{A})_p$. 
\subsection{Presheaf of cochain complex}
Let $\mathcal{C}$ be a presheaf on a locally geometric poset $\mathfrak{L}$ in a general sense.
\begin{defn}
	We call $\mathcal{C}$ is a \textbf{presheaf of cochain complex} if for $p\in \mathfrak{L}$, $\mathcal{C}_p=\bigoplus_i \mathcal{C}_p^i$ is a cochain complex with the differential
	$$\delta : \mathcal{C}_p^i \rightarrow \mathcal{C}_p^{i+1} $$
	and $f_{p,q}$ is cochain map between cochain complexes.
\end{defn}

We have seen from  Definition \ref{OS complex} that $A^*(\mathfrak{L},\mathcal{C})$ is a  complex  with differential $\partial$, where $\partial$ is induced by the differential on $A^*([0,p])$. Now assume that $\mathcal{C}$ is a presheaf of cochain complex. Then we
can use the differential on $\mathcal{C}$ to define another differential on $A^*(\mathfrak{L},\mathcal{C})$ by $$\delta(x\otimes c)= (-1)^{r(p)}x\otimes\delta(c)$$ where $x\in A^*([0,p])_p, c\in \mathcal{C}_p^i$ and $\delta$ in the right side is the differential of complex $\mathcal{C}_p$. This means that $A^*(\mathfrak{L},\mathcal{C})$ becomes  a double complex.
\vskip .2cm
Furthermore, consider the operation
$\partial+ \delta$ on $A^*(\mathfrak{L},\mathcal{C})$. We claim that $\partial+ \delta$ is a differential on $A^*(\mathfrak{L},\mathcal{C})$.
It suffices to check that $(\delta\partial + \partial\delta)(x\otimes c)=0$ for $x\in A^*([\mathbf{0},p])_p$ and $c \in \mathcal{C}_p^i$.
In fact, let $\partial x=\sum_j x_j$ as in Definition \ref{OS complex} where $x_j\in A^*([\mathbf{0},p])_{p_j}$ and $p$ covers $p_j$. Then $$\delta\partial(x\otimes c)= (-1)^{r(p)-1}\sum_j x_j\otimes f_{p,p_j}\delta (c)$$ and $$\partial\delta(x\otimes c)= (-1)^{r(p)}\sum_j x_j\otimes f_{p,p_j}\delta (c)$$ by definition, from which follows that  $(\delta\partial + \partial\delta)(x\otimes c)=0$ as desired. As well-known,  $A^*(L,\mathcal{C})$ with $\partial+ \delta$ is called the \textbf{associated total complex}, denoted by $Tot(A^*(L,\mathcal{C}))$, and $\partial+ \delta$ is called the \textbf{total differential}, which is of degree $+1$. In addition, for $x\in A^*([\mathbf{0},p])_p$ and $c \in \mathcal{C}_p^i$, the total degree of $x\otimes c$ is defined as $\deg(x\otimes c)=i-r(p)$.

\vskip .2cm
 Combining the above arguments, we have that

\begin{prop}
Let $\mathcal{C}$ be a presheaf of cochain complex on a locally geometric poset $\mathfrak{L}$. Then
	$A^*(\mathfrak{L},\mathcal{C})$ naturally admits a double complex structure with two differential $\partial$ and $\delta$, and  the associated total complex  $Tot(A^*(\mathfrak{L},\mathcal{C}))$ is also a cochain complex with the total differential $\partial+ \delta$ of degree $+1$.
\end{prop}

It is well known that for every double complex, there are two filtrations and two spectral sequences associated with it. For the double complex $A^*(\mathfrak{L},\mathcal{C})$, we choose the "column-wise" filtration
\begin{equation}\label{col filt}
\tau^{-k}Tot(A^*(\mathfrak{L},\mathcal{C}))=\bigoplus_{q\text{ with }r(q)\leq k} A^*(\mathfrak{L},\mathcal{C})_q
\end{equation}
It is a decreasing filtration of modules, satisfying the condition  of \cite[Theorem  2.6]{McCleary2000}. Thus we have
\begin{cor}\label{SS0}
	The spectral sequence associated with $\tau^{-*}$ satisfying
	$$E_2^{-i,j}=H^{-i}(A^*(\mathfrak{L},H^{j}(\mathcal{C})),\partial)$$
	which converges to $H^*(Tot(A^*(\mathfrak{L},\mathcal{C})))$ as modules.
\end{cor}

\subsection{Double complex model of manifold arrangements}\

For a manifold arrangement $\mathcal{A}$ in $M$, clearly $\mathcal{C}(\mathcal{A})$ is a presheaf of cochain complex. Then $A^*(\mathfrak{L},\mathcal{C}(\mathcal{A}))$ has a double complex structure. Furthermore, $A^*([p,\infty),\mathcal{C}(\mathcal{A}))$ is also a double complex for any $p\in \mathfrak{L}$ since $[p,\infty)$ is also a locally geometric poset.

The results in this subsection is not new, covered by \cite{Tosteson2016} or \cite{Petersen2017}, but it is written here for a local completeness so as to consider the case with multiplicative structure in Section 5 by doing an expansion. 

\begin{thm}\label{main}
	$H^*(Tot(A^*([p,\infty),\mathcal{C}(\mathcal{A}))))$ and $H^*(N(S_p),N(S_p)_0)$ are isomorphic as module.
\end{thm}
\begin{proof} Consider the filtration of $\mathcal{C}(\mathcal{A})$ in Definitin~\ref{presheaf of SS}.
	Denote $F^i\mathcal{C}(\mathcal{A}): p\longmapsto F^i\mathcal{C}(\mathcal{A})_p$. Then  $F^i\mathcal{C}(\mathcal{A})$ is also a presheaf of cochain complex on $\mathfrak{L}$ since the
	inclusion map $f_{p,q}$ is compatible with the filtration $F^i\mathcal{C}(\mathcal{A})_p$. Moreover, we may
 use $F^i\mathcal{C}(\mathcal{A})$ to give  a filtration $Tot(A^*([p,\infty),F^i\mathcal{C}(\mathcal{A})))$ of $Tot(A^*([p,\infty),\mathcal{C}(\mathcal{A})))$, also denoted by 
 $F^iTot(A^*([p,\infty),$ $\mathcal{C}(\mathcal{A})))$.

 \vskip .2cm Now we calculate $H^*(Tot(A^*([p,\infty),\mathcal{C}(\mathcal{A}))))$ by this filtration. By Definitin~\ref{presheaf of SS}, we see that $$F^i\mathcal{C}(\mathcal{A})_p /F^{i+1}\mathcal{C}(\mathcal{A})_p= C^*(F_p^{i+1}M,F_p^{i}M)$$ which is also a presheaf of cochain complex on $\mathfrak{L}$ when $p$ runs over $\mathfrak{L}$.  Then $E_0$-term is
$$F^iTot(A^*([p,\infty),\mathcal{C}(\mathcal{A})))/F^{i+1}Tot(A^*([p,\infty),\mathcal{C}(\mathcal{A})))=Tot(A^*([p,\infty),F^i\mathcal{C}(\mathcal{A}) /F^{i+1}\mathcal{C}(\mathcal{A})))$$
We may calculate the homology of this total complex by "calculating homology twice" as seen in \cite[Theorem 2.15]{McCleary2000}.  Firstly let us calculate the homology under differential $\delta$:
	$$H^{i+j}(A^*([p,\infty),F^i\mathcal{C}(\mathcal{A}) /F^{i+1}\mathcal{C}(\mathcal{A})),\delta)=A^*([p,\infty),E_{1,*}^{i,j}(\mathcal{A}))$$
	where $E_{1,*}^{i,j}(\mathcal{A}) \cong \bigoplus_{r(\alpha)=i} j_{\alpha *}H^{i+j}(N(S_\alpha),N(S_\alpha)_0)$ is the direct sum of some sky-scraper presheaves as we calculated in Theorem \ref{presheaf of E1}. Secondly we calculate the homology under differential $\partial$ as follows: $H^{-*}(A^*([p,\infty),E_{1,*}^{i,j}(\mathcal{A})),\partial)=0$ for all $i\neq r(p)$  by Lemma \ref{sky}, which means  that $H^*(Tot(A^*([p,\infty),F^i\mathcal{C}(\mathcal{A}) /F^{i+1}\mathcal{C}(\mathcal{A}))))$ vanishes for all $i\neq r(p)$; in other words, the $E_1^{i,j}$-term of filtration $F^iTot(A^*([p,\infty),\mathcal{C}(\mathcal{A})))$ vanishes for all $i\neq r(p)$. So the quotient map $$Tot(A^*([p,\infty),F^{r(p)}\mathcal{C}(\mathcal{A})))\rightarrow Tot(A^*([p,\infty),F^{r(p)}\mathcal{C}(\mathcal{A})/F^{r(p)+1}\mathcal{C}(\mathcal{A})))$$ induces an isomorphism of cohomologies.  When restricted on $[p,\infty)$, since $F^{r(p)}\mathcal{C}(\mathcal{A})_\alpha$ always equals $C^*(M,M-M_p)$ for all $\alpha\geq p$ by the definition of $F^*\mathcal{C}(\mathcal{A})$, we see that $F^{r(p)}\mathcal{C}(\mathcal{A})/F^{r(p)+1}\mathcal{C}(\mathcal{A})$ can only  have  nonzero elements of $C^*(F^{r(p)+1}_pM,F^{r(p)}_pM)$ on $p$,  and zero otherwise. So $$Tot(A^*([p,\infty),F^{r(p)}\mathcal{C}(\mathcal{A})/F^{r(p)+1}\mathcal{C}(\mathcal{A})))=
(C^*(F^{r(p)+1}_pM,F^{r(p)}_pM),\delta).$$
Combining this equality and above quotient map, we conclude that the map $$Tot(A^*([p,\infty),\mathcal{C}(\mathcal{A})))\rightarrow C^*(F^{r(p)+1}_pM,F^{r(p)}_pM)$$ induces an isomorphism of cohomologies. On the other hand, we know by Remark \ref{rem of excision} that $H^*(F^{r(p)+1}_pM,F^{r(p)}_pM)\cong H^*(N(S_p),N(S_p)_0)$. This completes the proof.
\end{proof}

\begin{rem}
	Theorem~\ref{main} gives the equivalence expression of  $H^*(N(S_p),N(S_p)_0)$ for every $p$  rather than  only $H^*(\mathcal{M}(\mathcal{A}))$.
\end{rem}

Let $p=\textbf{0}$ be minimal element in $\mathfrak{L}$. Then we have following corollary
\begin{cor}\label{main cor}
$ H^*(Tot(A^*(\mathfrak{L},\mathcal{C}(\mathcal{A}))))$ and $H^*(\mathcal{M}(\mathcal{A}))$ are isomorphic as modules. In particular, this isomorphism is actually induced by the quotient map $A^*(\mathfrak{L},\mathcal{C}(\mathcal{A})) \rightarrow C^*(\mathcal{M}(\mathcal{A}))$ that maps $A^*(\mathfrak{L},\mathcal{C}(\mathcal{A}))_q$ to zero for $q>\mathbf{0}$ and  $A^*(\mathfrak{L},\mathcal{C}(\mathcal{A}))_\mathbf{0}=C^*(M)$ to $C^*(\mathcal{M}(\mathcal{A}))$ (naturally induced by the inclusion $\mathcal{M}(\mathcal{A}) \hookrightarrow M$).
\end{cor}
\begin{proof}
	Let $p=\textbf{0}$ in the proof of Theorem~\ref{main}. We note that $N(S_{\textbf{0}})=S_{\textbf{0}}=\mathcal{M}(\mathcal{A})$, $N(S_{\textbf{0}})_0=\emptyset$, $F^{1}_{\textbf{0}}M=\mathcal{M}(\mathcal{A})$, and $F^{0}_{\textbf{0}}M=\emptyset$. Then the quotient map in the proof of Theorem~\ref{main} $$Tot(A^*([p,1],F^{r(p)}\mathcal{C}(\mathcal{A})))\rightarrow Tot(A^*([p,1],F^{r(p)}\mathcal{C}(\mathcal{A})/F^{r(p)+1}\mathcal{C}(\mathcal{A})))$$ becomes $$Tot(A^*(L,\mathcal{C}(\mathcal{A})))\rightarrow Tot(A^*(L,\mathcal{C}(\mathcal{A})/F^{1}\mathcal{C}(\mathcal{A}))).$$
Now, since  the quotient presheaf $\mathcal{C}(\mathcal{A})_p/F^{1}\mathcal{C}(\mathcal{A})_p$ is zero if $p>\textbf{0}$ and be $C^*(\mathcal{M}(\mathcal{A})) $ if $p=\mathbf{0}$, we see that the above quotient map would be a zero-morphism if  $p>\mathbf{0}$ and the quotient map $C^*(M) \rightarrow C^*(\mathcal{M}(\mathcal{A}))$ induced by the inclusion $\mathcal{M}(\mathcal{A}) \hookrightarrow M$ if  $p=\mathbf{0}$.
\end{proof}

Consider the spectral sequence associated with the double complex $A^*([p,\infty),\mathcal{C}(\mathcal{A}))$, we have an immediate corollary.
\begin{cor}\label{immed cor}
Associated with double complex $A^*([p,\infty),\mathcal{C}(\mathcal{A}))$, there is a spectral sequence  with $$E_1^{-i,j}=A^*([p,\infty),H^{j}(\mathcal{A}))_{-i}$$
 $$E_2^{-i,j}=H^{-i}(A^*([p,\infty),H^{j}(\mathcal{A})),\partial)$$ converges to $H^*(N(S_p),N(S_p)_0)$ as modules. In particular, for $p=\textbf{0}$, there exists a spectral sequence with  $$E_2^{-i,j}=H^{-i}(A^*(\mathfrak{L},H^{j}(\mathcal{A})),\partial)$$ converges to $H^*(\mathcal{M}(\mathcal{A}))$ as modules. Recall that $H^*(\mathcal{A})_p:=H^*(M,M-M_p)$.
\end{cor}

Using the approach developed in this section, we can easily reprove Zaslavsky's result \cite{Zaslavsky1975} about the $f$-polynomials of hyperplane arrangements in $M=\mathbb{R}^d$.
\begin{cor}\label{f-poly}
	Let $\mathcal{A}$ be a central hyperplane arrangement in $\mathbb{R}^d$ with intersection lattice $L$. Then the number of $k$-faces equals $$\sum_{p\in L,r(p)=d-k}\dim(A^*([p,\mathbf{1}]))$$ where $A^*([p,\mathbf{1}])$ is the canonical OS-algebra of $[p,\mathbf{1}]$ with any field coefficients.
\end{cor}
\begin{proof}
	Choose quasi-intersection poset to be intersection lattice of $\mathcal{A}$. Every $k$-face is a region of some $S_p$ with $r(p)=d-k$.  Then the number of $k$-faces equals\\ $\sum_{p\in L,r(p)=d-k}\dim(H^*(N(S_p),N(S_p)_0))$ by Thom isomorphism. Since $M_p=\mathbb{R}^{d-r(p)}$, the structure map $f_{p,q}^*$ of $H^*(\mathcal{A})$ is zero and the spectral sequence in Corollary \ref{immed cor} collapses on $E^1$-term because of the dimensional  reason, so $\dim(H^*(N(S_p),N(S_p)_0))=\dim(A^*([p,\mathbf{1}]))$.
\end{proof}

\begin{rem}
	Above corollary agrees with the original description of Zaslavsky \cite{Zaslavsky1975}. In this view point,  Theorem \ref{main} essentially considers not only "regions (or chambers)" but also the information of all "faces" of a manifold arrangement $\mathcal{A}$, so it can be regarded as  a topological generalization of Zaslavsky's $f$-polynomial.
\end{rem}

\subsection{The relationship with mixed Hodge structure}\label{hodge}~\\

 In \cite{Petersen2017}, Petersen obtained a spectral sequence of mixed Hodge structures in the complex algebraic setting.
In this subsection, we discuss the mixed Hodge structure of our model, by a different approach.

\vskip .2cm
Let $R$ be a noetherian subring of $\mathbb{C}$ such that $R\otimes \mathbb{Q}$ is a field. A Hodge $R$-complex $K$ of weight $n$ is a system $$[K_R,(K_{\mathbb{C}},F),\alpha:K_R\otimes\mathbb{C}\simeq K_{\mathbb{C}}]$$ where $K_R$ is a $R$-complex with finite type cohomology, $K_{\mathbb{C}}$ is a $\mathbb{C}$-complex with decreasing filtration $F$ and $\alpha$ is an isomorphism in $D^+(\mathbb{C})$, satisfying: (1) the differential of $K_{\mathbb{C}}$ is strictly compatible with filtration $F$; (2) the induced filtration F on $H^k(K_{\mathbb{C}}) $ defines a pure Hodge structure of weight $n+k$.
\vskip .2cm
A $R$-mixed Hodge complex (MHC) $K$ is a system
$$[K_R;(K_{R\otimes\mathbb{Q}},W),\alpha:K_R\otimes\mathbb{Q}\simeq K_{R\otimes\mathbb{Q}};(K_{\mathbb{C}},W,F),\beta:(K_{R\otimes\mathbb{Q}},W)\otimes\mathbb{C}\simeq (K_{\mathbb{C}},W)]$$ where $K_R$ is a $R$-complex with finite type cohomology, $K_{R\otimes\mathbb{Q}}$ is a $R\otimes\mathbb{Q}$-complex with increasing filtration $W$ and $\alpha$ is an isomorphism in derived category of $R\otimes\mathbb{Q}$-module, $K_{\mathbb{C}}$ is a $\mathbb{C}$-complex with increasing filtration $W$ and decreasing filtration $F$ and $\beta$ is an isomorphism in derived category of filtrated $\mathbb{C}$-module, satisfying: the $n$-th graded piece $Gr_n^WK=[Gr_n^WK_{R\otimes\mathbb{Q}},Gr_n^W(K_{\mathbb{C}},F),Gr_n^W\beta]$ is a $R\otimes\mathbb{Q}$-Hodge complex of weight $n$ for every $n$.
\vskip .2cm
There is a functorial MHC associated with any smooth complex algebraic variety $U$ (see \cite[Theorem 4.8]{Zein2017}) , denoted by $K(U)$, which induces the canonical mixed Hodge structure on $H^*(U)$. For more details about MHC and the definition of morphism of MHC, we refer to \cite{Board2008}.
\vskip .2cm
The model $Tot(A^*(\mathfrak{L},\mathcal{C}))$ will be a mixed Hodge complex if $\mathcal{C}$ is a presheaf of mixed Hodge complex, i.e., all $f_{p,q}:\mathcal{C}_p \rightarrow \mathcal{C}_q$ are morphisms of mixed Hodge complex.
\begin{defn}
	Assume $\mathcal{C}$ is a presheaf of mixed Hodge complex. Namely
	$$\mathcal{C}=[\mathcal{C}_{R};(\mathcal{C}_{R\otimes\mathbb{Q}},W),\alpha:\mathcal{C}_{R}\otimes\mathbb{Q}\simeq \mathcal{C}_{R\otimes\mathbb{Q}};(\mathcal{C}_{\mathbb{C}},W,F),\beta:(\mathcal{C}_{R\otimes\mathbb{Q}},W)\otimes\mathbb{C}\simeq (\mathcal{C}_{\mathbb{C}},W)]$$
	where $\mathcal{C}_{R}:p\longmapsto \mathcal{C}_{pR}$ (resp. $\mathcal{C}_{R\otimes\mathbb{Q}},\mathcal{C}_{\mathbb{C}}$) is a presheaf of $R$ (resp. $R\otimes\mathbb{Q},\mathbb{C}$) module and $f_{p,q}:\mathcal{C}_p \rightarrow \mathcal{C}_q$ is the morphism of mixed Hodge complex. Define $Tot(A^*(\mathfrak{L},\mathcal{C}))$ be a system
	$$[Tot(A^*(\mathfrak{L},\mathcal{C}_{R}));(Tot(A^*(\mathfrak{L},\mathcal{C}_{R\otimes\mathbb{Q}})),W),\tilde{\alpha};(Tot(A^*(\mathfrak{L},\mathcal{C}_{\mathbb{C}})),W,F),\tilde{\beta}]$$ where $$W_m Tot(A^*(\mathfrak{L},\mathcal{C}_{R\otimes\mathbb{Q}}))=\bigoplus_p A^*([0,p])_p \otimes W_{m-r(p)}\mathcal{C}_{pR\otimes\mathbb{Q}}$$
	$$F_m Tot(A^*(\mathfrak{L},\mathcal{C}_{\mathbb{C}}))=\bigoplus_p A^*([0,p])_p \otimes F_{m}\mathcal{C}_{p\mathbb{C}}$$
	and $\tilde{\alpha},\tilde{\beta}$ is the isomorphism induced by $\alpha,\beta$.
	
\end{defn}

\begin{prop}\label{ind hodge}
	If $\mathcal{C}$ is a presheaf of mixed Hodge complex, then $Tot(A^*(\mathfrak{L},\mathcal{C}))$ in above definition is a well defined mixed Hodge complex.
\end{prop}
\begin{proof}
	A direct calculation gives that the $n$-th graded piece $Gr_m^WTot(A^*(\mathfrak{L},\mathcal{C}))$ equals $\bigoplus_p A^*(\mathfrak{L})_p \otimes Gr_{m-r(p)}^W\mathcal{C}_p$ with differential induced by $\delta$ on every $\mathcal{C}_p$ (the contribution of $\partial$ to the differential vanishes since $\partial$ preserve $W$), notice every element $\bullet\otimes\mathcal{C}_p^i$ have total degree $i-r(p)$, so $A^*(\mathfrak{L})_p \otimes Gr_{m-r(p)}^W\mathcal{C}_p$ is a weight $m$ Hodge complex for any $p \in \mathfrak{L}$.
\end{proof}

Calculate the spectral sequence associated with weight filtration of $Tot(A^*(\mathfrak{L},\mathcal{C}))$ carefully, we have $E_1^{-p,q}=\bigoplus_s A^*([0,s])_s \otimes E_1^{-p+r(s),q}(\mathcal{C}_s)$, and the complex of $E_1$ page $$\cdots\rightarrow  E_1^{-p,q}\rightarrow E_1^{-p+1,q}\rightarrow\cdots$$ equals $Tot(A^*(\mathfrak{L},E_1^{*,q}(\mathcal{C})))$, where $E_1^{*,q}(\mathcal{C})_s$ is the $E_1$ page of weight filtration of every $\mathcal{C}_s$.
Notice that $E_2^{-p,q}=Gr_q^{W}H^{q-p}(K_{R\otimes\mathbb{Q}})$ for every MHC $K$ \cite[Theorem 3.18]{Board2008}, we have
\begin{cor}
	$$Gr_q^{W}H^{q-p}(Tot(A^*(\mathfrak{L},\mathcal{C}_{R\otimes\mathbb{Q}})))=
	H^{-p}(Tot(A^*(\mathfrak{L},E_1^{*,q}(\mathcal{C}_{R\otimes\mathbb{Q}}))))$$
\end{cor}
The total complex $Tot(A^*(\mathfrak{L},E_1^{*,q}(\mathcal{C}_{R\otimes\mathbb{Q}})))$ in the right side will be simpler if $\mathcal{C}$ satisfies some "purity" condition, for example:
\begin{cor}\label{pure result}
	If every $H^k(\mathcal{C}_s)$ is pure of weight $k$, then
	$$Gr_q^{W}H^{q-p}(Tot(A^*(\mathfrak{L},\mathcal{C}_{R\otimes\mathbb{Q}})))= H^{-p}(A^*(\mathfrak{L},H^q(\mathcal{C}_{R\otimes\mathbb{Q}})),\partial).$$
\end{cor}
\begin{proof}
	In this case, $E_2^{p,q}(\mathcal{C}_{R\otimes\mathbb{Q}}) = 0$ for any $p\neq 0$ and $E_2^{0,q}(\mathcal{C}_{R\otimes\mathbb{Q}}) = H^q(\mathcal{C}_{R\otimes\mathbb{Q}})$. Then we can calculate $H^*(Tot(A^*(\mathfrak{L},E_1^{*,q}(\mathcal{C}_{R\otimes\mathbb{Q}}))))$ by "calculate cohomology twice", that is, $H^{-p}(Tot(A^*(\mathfrak{L},E_1^{*,q}(\mathcal{C}_{R\otimes\mathbb{Q}}))))
=H^{-p}(A^*(\mathfrak{L},H^q(\mathcal{C}_{R\otimes\mathbb{Q}})),\partial)$ as desired. 
\end{proof}

Now, suppose that $M$ and every $N_i\in \mathcal{A}$ are complex smooth algebraic varieties. We are going  to construct a presheaf of mixed Hodge complex as a substitute for $\mathcal{C}(\mathcal{A})$ in Theorem \ref{main}. Recall that $\mathcal{C}(\mathcal{A})_p= C^*(M,M-M_p)$. One possible choice is the substitute $\mathcal{K}(\mathcal{A})_p=Cone(K(M) \rightarrow K(M-M_p))[-1]$ for $\mathcal{C}(\mathcal{A})_p$, where $K(\bullet)$ is the associated MHC of any variety and $Cone$ denotes the mixed cone of a map. This simple construction does not behave very well since the quotient map $Tot(A^*(\mathfrak{L},\mathcal{C}(\mathcal{A}))) \rightarrow C^*(\mathcal{M}(\mathcal{A}))$ in Corollary \ref{main cor} is no longer a chain map of complexes if we replace $\mathcal{C}$ by $\mathcal{K}$. We will use the technique of mapping telescope (cone) to fix this problem.

\vskip .2cm

Recall that there is a mixed cone $Cone(\phi)$ for any morphism $\phi: K \rightarrow L$ of MHC $K,L$. Roughly speaking, $Cone(\phi)= K[1] \oplus L$ with $d(x,y)=(dx,f(x)+dy)$ and filtrations $W[1]\oplus W$ and $F\oplus F$. Actually $Cone(\phi)$ is a MHC with these filtrations,  we refer to \cite[Definition 3.31]{Zein2017} for strict definition about mixed cone.

\vskip .2cm
We define our {mixed telescope cone} by using mixed cone as follows.
\begin{defn}
	Let $K_1 \stackrel{f_1}{\rightarrow} K_2 \stackrel{f_2}{\rightarrow}\cdots \stackrel{f_{n-1}}{\rightarrow} K_n $ be a sequence of map of MHC and $n\geq 2$. Construct a function $\tilde{f}: \bigoplus_1^{n-1} K_i \rightarrow \bigoplus_2^{n} K_i$ as follows $$\tilde{f}(x_1,x_2,\ldots,x_{n-1})=(x_2+f_1(x_1),\ldots,x_{n-1}+f_{n-2}(x_{n-2}),f_{n-1}(x_{n-1})).$$
We define 	the \textbf{mixed telescope cone} of the sequence $K_1 \stackrel{f_1}{\rightarrow} K_2 \stackrel{f_2}{\rightarrow}\cdots \stackrel{f_{n-1}}{\rightarrow} K_n $ as the mixed cone $Cone(\tilde{f})$, also denoted by   $TCone(K_1 \stackrel{f_1}{\rightarrow} K_2 \stackrel{f_2}{\rightarrow}\cdots \stackrel{f_{n-1}}{\rightarrow} K_n)$. This  coincides with the usual definition of mixed cone if $n=2$.	
\end{defn}
\begin{rem}
	It is easy to check there is a short exact sequence of MHC
	{\small $$0 \rightarrow TCone(K_1 \stackrel{f_1}{\rightarrow} \cdots \stackrel{f_{s-1}}{\rightarrow} K_s) \rightarrow TCone(K_1 \stackrel{f_1}{\rightarrow} \cdots \stackrel{f_{n-1}}{\rightarrow} K_n) \rightarrow TCone(K_s \stackrel{f_s}{\rightarrow} \cdots \stackrel{f_{n-1}}{\rightarrow} K_n) \rightarrow 0.$$} 
If $K_1 \stackrel{f_1}{\rightarrow} \cdots \stackrel{f_{n-1}}{\rightarrow} K_n$ is the associated MHC of $U_n\subseteq \cdots \subseteq U_1$, then the long exact sequence induced by above short exact sequence is just the long exact sequence of triple $(U_1,U_s,U_n)$.
\end{rem}

\begin{defn}
	Define $$\mathcal{K}(\mathcal{A})_p= TCone(K(M)\rightarrow\cdots \rightarrow K(F_p^iM)\rightarrow\cdots\rightarrow K(M-M_p))[-1]$$ where $F_p^iM$ is the filtration in Definition \ref{geo filt}.
\end{defn}
\begin{prop}
	 The $\mathcal{K}(\mathcal{A})$ in above definition is a presheaf of mixed Hodge complex, and there is a quotient map   $$Tot(A^*(\mathfrak{L},\mathcal{K}(\mathcal{A}))) \rightarrow Cone(\mathcal{K}(F^{r(p)+1}_pM)\rightarrow \mathcal{K}(F^{r(p)}_pM))[-1]$$ 
which is a morphism of MHC and induces an isomorphism with mixed Hodge structure $$H^*(Tot(A^*([p,\infty),\mathcal{K}(\mathcal{A}))))\cong H^*(F^{r(p)+1}_pM,M-M_p).$$
	 In particular, for $p=\mathbf{0}$, $$H^*(Tot(A^*(\mathfrak{L},\mathcal{K}(\mathcal{A}))))\cong H^*(\mathcal{M}(\mathcal{A}))$$ with mixed Hodge structure.
\end{prop}
\begin{proof}
	Firstly, we need to define map $f_{p,q}: \mathcal{K}(\mathcal{A})_p \rightarrow \mathcal{K}(\mathcal{A})_q$. The $F_q^iM$ is always a subspace of $F_p^iM$ for any $p\geq q$ by definition. Then these maps $K(F_p^iM) \rightarrow K(F_q^iM)$ induce map of telescope cone $f_{p,q}:\mathcal{K}(\mathcal{A})_p \rightarrow \mathcal{K}(\mathcal{A})_q$. It is obvious that $\mathcal{K}(\mathcal{A})$ is a presheaf of mixed Hodge complex with these $f_{p,q}$.
	
	In the proof of Theorem \ref{main}, we use a filtration of $\mathcal{C}(\mathcal{A})$ such that $E_1$-term is sky-scraper presheaf. Similarly, let
	\begin{subnumcases}
	{F^i\mathcal{K}(\mathcal{A})_p=}
	TCone(K(M)\rightarrow\cdots\rightarrow K(F_p^*M)\rightarrow\cdots\rightarrow K(M-M_p))[-1], &if \ ~$i < r(p)$ \nonumber\\
	TCone(K(M)\rightarrow\cdots \rightarrow K(F_p^iM))[-1], &if \ ~$i \geq r(p).$ \nonumber
	\end{subnumcases}
	It is easy to check $F^i\mathcal{K}(\mathcal{A})$ is a filtration of presheaf ($f_{p,q}$ preserves this filtration). There is no difference between the homologies of $F^i\mathcal{K}(\mathcal{A})_p$ and $F^i\mathcal{C}(\mathcal{A})_p$, i.e.,  they are isomorphic to  $H^*(M,F^i_pM)$. So all processes in the proof of theorem \ref{main} can still be carried out very well for the filtration $F^i\mathcal{K}(\mathcal{A})_p$. Notice that 
$$F^{r(p)}\mathcal{K}(\mathcal{A})_q/F^{r(p)+1}\mathcal{K}(\mathcal{A})_q \cong 
\begin{cases} Cone(\mathcal{K}(F^{r(p)+1}_pM)\rightarrow \mathcal{K}(F^{r(p)}_pM))[-1] & \text{ if }p=q\\
0 &  \text{ if } q>p.
\end{cases}
$$ 
	Then we have that\\
	$$H^*(Tot(A^*([p,\infty),\mathcal{K}(\mathcal{A}))))\cong H^*(F^{r(p)+1}_pM,M-M_p)$$ with mixed Hodge structure.
\end{proof}

In particular, if $M$ is projective, we know that $H^k(\mathcal{A})_p=H^k(M,M-M_p)$ is pure of weight $k$ (actually, it equals a Tate twist of $H^*(M_p)$). Applying Corollary \ref{pure result}, we have
\begin{cor}\label{alg arg}
	Assume  $M$ is projective and using $\mathbb{Q}$ coefficients, we have
	 $$Gr_n^WH^{n-i}(\mathcal{M}(\mathcal{A})) =H^{-i}(A^*(\mathfrak{L},H^{n}(\mathcal{A})),\partial)$$
\end{cor}
At the end of this subsection, we shall compare the weight filtration with the "column-wise" filtration (\ref{col filt}) $\tau^{-*}$ of $Tot(A^*(\mathfrak{L},\mathcal{K}(\mathcal{A})))$.
\begin{thm}\label{compare thm}
	Assume that $M$ is projective and we use $\mathbb{Q}$ as coefficients. Then the "column-wise" filtration $\tau^{-*}$ and the weight filtration $W_*$ on $Tot(A^*(\mathfrak{L},\mathcal{K}(\mathcal{A})))$ induce the same filtration on cohomology group $H^*(Tot(A^*(\mathfrak{L},\mathcal{K}(\mathcal{A}))))\cong H^*(\mathcal{M}(\mathcal{A}))$.
\end{thm}
\begin{proof}
	Assume $\mathcal{C}$ is a presheaf of MHC and every $H^k(\mathcal{C}_s)$ is pure of weight $k$. We prove a more general version that the "column-wise" filtration $\tau^{-*}$ and the weight filtration $W_*$ on $Tot(A^*(\mathfrak{L},\mathcal{C}))$ induce the same filtration on cohomology $H^*(Tot(A^*(\mathfrak{L},\mathcal{C})))$. 
	\vskip .2cm
	Recall that $\tau^{-i}Tot(A^*(\mathfrak{L},\mathcal{C}))= \bigoplus_{p\in\mathfrak{L},r(p)\leq i}A^*([0,p])_p \otimes \mathcal{C}_p= Tot(A^*(\{p\in\mathfrak{L} | r(p)\leq i\},\mathcal{C}))$. We need to show that $\IM(H^*(\tau^{-i}Tot(A^*(\mathfrak{L},\mathcal{C})))\rightarrow H^*(Tot(A^*(\mathfrak{L},\mathcal{C}))))$ equals $W_i H^*(Tot(A^*(\mathfrak{L},\mathcal{C})))$.
	\vskip .2cm
	Corollary \ref{pure result} shows that
	\begin{equation}\tag{1}
	Gr_{k+j}^WH^k(\tau^{-i}Tot(A^*(\mathfrak{L},\mathcal{C})))=H^{-j}(A^*(\{p\in\mathfrak{L} | r(p)\leq i\},H^{k+j}\mathcal{C}),\partial)
	\end{equation}
	\begin{equation}\tag{2} 
	Gr_{k+j}^WH^k(Tot(A^*(\mathfrak{L},\mathcal{C})))=H^{-j}(A^*(\mathfrak{L},H^{k+j}\mathcal{C}),\partial)
	\end{equation}
	 , so the inclusion map induces a surjective map on $Gr_{k+j}^WH^k=W_j/W_{j-1}$ for $j=i$ and an isomorphism for $j<i$. Then it  induces a surjective map on $W_i$ as well and therefore $\IM(W_iH^*(\tau^{-i}Tot(A^*(\mathfrak{L},\mathcal{C})))\rightarrow H^*(Tot(A^*(\mathfrak{L},\mathcal{C}))))$ equals $W_i H^*(Tot(A^*(\mathfrak{L},\mathcal{C})))$.
	\vskip .2cm
	Now, we only need to check $W_iH^*(\tau^{-i}Tot(A^*(\mathfrak{L},\mathcal{C})))$ equals $ H^*(\tau^{-i}Tot(A^*(\mathfrak{L},\mathcal{C})))$ itself. Observe that $Gr_{k+j}^WH^k(\tau^{-i}Tot(A^*(\mathfrak{L},\mathcal{C})))=0$ for $j>i$ by equation (1), which means that $W_j/W_{j-1}=0$ for $j>i$, so   $W_iH^*(\tau^{-i}Tot(A^*(\mathfrak{L},\mathcal{C})))$ equals $H^*(\tau^{-i}Tot(A^*(\mathfrak{L},\mathcal{C})))$ itself, which completes the proof.
\end{proof}

\begin{rem}
	Unlike the weight filtration $W_*$, the filtration $\tau^{-*}$ is also defined for our model $Tot(A^*(\mathfrak{L},\mathcal{C}(\mathcal{A})))$ with any coefficients where the manifolds in $\mathcal{A}$ need not to be varieties.
\end{rem}

\subsection{Inclusion map of sub-arrangements}\label{inc of sub}
In this section, we will consider the sub-arrangements $\mathcal{A}|p=\{M_{a_i}|r(a_i)=1,~a_i\leq p\}$. It is easy to see that every $\mathcal{A}|p$ is also a manifold arrangement with intersection lattice $[0,p]$. $\mathcal{M}(\mathcal{A})$ is obviously a subspace of $ \mathcal{M}(\mathcal{A}|p)$, so there is the inclusion  $i: \mathcal{M}(\mathcal{A})\hookrightarrow\mathcal{M}(\mathcal{A}|p)$. A natural question arises:
{\em what is the induced map  $i^*: H^*(\mathcal{M}(\mathcal{A}|p))\rightarrow H^*(\mathcal{M}(\mathcal{A}))$ under the isomorphism in Theorem \ref{main}?}

\vskip .2cm

For an OS-algebra $A^*(L)$, we know that there is an inclusion map $A^*([\mathbf{0},p])\rightarrow A^*(L)$ for every $p\in L$, which  maps $A^*([\mathbf{0},p])_q$ isomorphically onto $A^*(L)_q$ for every $q\leq p$. Then we have an inclusion map  $j: A^*([\mathbf{0},p],\mathcal{C}(\mathcal{A}))\rightarrow A^*(L,\mathcal{C}(\mathcal{A}))$ as double complexes.

\begin{prop}\label{sub arrangement}
	There is a commutative diagram
	\[
	\begin{CD}
	H^*(\mathcal{M}(\mathcal{A}|p)) @>i^*>> H^*(\mathcal{M}(\mathcal{A}))\\
	@AAA @AAA\\
	H^*(Tot(A^*([\mathbf{0},p],\mathcal{C}(\mathcal{A})))) @>j^*>> H^*(Tot(A^*(L,\mathcal{C}(\mathcal{A}))))
	\end{CD}
	\]
where $j^*$ is the induced map of cohomology by $j$, and   every column arrow is an isomorphism as in Corollary \ref{main cor}.
\end{prop}
\begin{proof}
	We see in the proof of Theorem \ref{main} that the isomorphism $H^*(Tot(A^*(L,\mathcal{C}(\mathcal{A}))))\rightarrow H^*(\mathcal{M}(\mathcal{A}))$ is induced by the quotient $A^*(L,\mathcal{C}(\mathcal{A}))\rightarrow C^*(\mathcal{M}(\mathcal{A}))$ by mapping $A^*(L)_q\otimes \mathcal{C}(\mathcal{A})_q$ to zero for all $q\neq \textbf{0}$ and mapping $\mathcal{C}(\mathcal{A})_0=C^*(M)$ to  $C^*(\mathcal{M}(\mathcal{A}))$. Now it suffices to check that the following diagram is commutative:
	\[
	\begin{CD}
	C^*(\mathcal{M}(\mathcal{A}|p)) @>>> C^*(\mathcal{M}(\mathcal{A}))\\
	@AAA @AAA\\
	A^*([\mathbf{0},p],\mathcal{C}(\mathcal{A})) @>>> A^*(L,\mathcal{C}(\mathcal{A})).
	\end{CD}
	\]
	Choose an element $\sum x_q\otimes c_q \in A^*([\mathbf{0},p],\mathcal{C}(\mathcal{A}))$ where $x_q\in A^*([\mathbf{0},p])_q$ and $c_q \in \mathcal{C}(\mathcal{A})_q$, the image of this element in $C^*(\mathcal{M}(\mathcal{A}))$ is the image of $x_{\textbf{0}}c_{\textbf{0}}$ under the quotient  $C^*(M)\rightarrow C^*(\mathcal{M}(\mathcal{A}))$ regardless of the 'path' we choose.
\end{proof}

\section{Product Structure}\label{prod str}

In this section, we are going to study the product structure on the double complex $A^*([p,\infty),\mathcal{C}(\mathcal{A}))$. We first discuss some general construction in subsection 5.1.  We will see that if $\mathcal{C}$ is a monoidal presheaf of DGA, then $Tot(A^*(\mathfrak{L},\mathcal{C}))$ is a differential algebra. Unfortunately,  $\mathcal{C}(\mathcal{A})_p=C^*(M,M-M_p)$ is not a monoidal presheaf under the cup product because the cup product $c_1\cup c_2$ for $c_1 \in C^*(M,M-M_p)$ and $c_2\in C^*(M,M-M_q)$ may not be contained in $\bigoplus_{s\in p\mathring{\vee}q}C^*(M,M-M_{s})$, where $C^*(M,M-M_p)$ and $C^*(M,M-M_q)$ are regarded as sub-complexes of $C^*(M)$.  We will modify $\mathcal{C}(\mathcal{A})$ into a monoidal presheaf $\hat{\mathcal{C}}(\mathcal{A})$ in subsection 5.2.

\subsection{Monoidal presheaf of DGA}\label{our double complex}\

\begin{defn}
	Let $\mathcal{C}$ be a presheaf of cochain complex. We call $\mathcal{C}$  a \textbf{monoidal presheaf of DGA} if $\mathcal{C}$ is  a monoidal presheaf with the monoidal product satisfying
	$$\delta(c_1c_2)=\delta(c_1)c_2+(-1)^{i}c_1\delta(c_2)$$
	for all $c_1 \in \mathcal{C}_p^i, c_2 \in \mathcal{C}_q^j$, where $\delta$ is the differential of every cochain complex $\mathcal{C}_p$.
\end{defn}

\begin{rem}
	For the monoidal product on $\mathcal{C}$, we have $\mathcal{C}_p \cdot \mathcal{C}_p \subset \mathcal{C}_p$ since $p\mathring{\vee} p=\{p\}$, so every $\mathcal{C}_p$ be a differential graded algebra. This is the reason why we use the name 'monoidal presheaf of DGA'.
\end{rem}

	Let $\mathcal{C}$ be a monoidal presheaf of DGA on locally geometric poset $\mathfrak{L}$.  Then the associated double complex $(A^*(\mathfrak{L},\mathcal{C}), \partial, \delta)$ has a natural product structure induced by the OS-algebra and the monoidal product of $\mathcal{C}$, as we defined in Definition \ref{OS-alg with coef}, that is	
\begin{equation} \label{product}
(x\otimes c_1)\cdot (y\otimes c_2)=(-1)^{\deg(c_1)r(q)}\sum_{s\in p\mathring{\vee} q}(i_{s}x\cdot i_{s}y)\otimes j_s(c_1\cdot c_2)
\end{equation} for $x\in A^*([\mathbf{0},p])_p,y\in A^*([\mathbf{0},q])_q, c_1 \in \mathcal{C}_p, c_2 \in \mathcal{C}_q$, where product "$\cdot$" on the right side is the product of OS-algebra $A^*([\mathbf{0},s])$, $j_s$ is projection on $\mathcal{C}_s$ and $i_s$ is the imbedding map in Proposition \ref{OS-property}(3).
Furthermore, $\partial$ and $\delta$ are both derivation of algebra $A^*(\mathfrak{L},\mathcal{C})$, we need a lemma before prove this property.

\begin{lem}\label{a simple lem}
	The following equations hold in $A^*(\mathfrak{L},\mathcal{C})$
	$$\sum_{t\in q\mathring{\vee} s}x_t\otimes j_t(b\cdot f_{p,q}a)= \sum_{k\in p\mathring{\vee} s}x_{\lambda k}\otimes f_{k,\lambda k}(j_{k}(b\cdot a))$$
and
	$$\sum_{t\in q\mathring{\vee} s}x_t\otimes j_t((f_{p,q}a) \cdot b)= \sum_{k\in p\mathring{\vee} s}x_{\lambda k}\otimes f_{k,\lambda k}(j_{k}(a\cdot b))$$
	where $p,q,s\in \mathfrak{L}, q\leq p$, $x_t\in A^*([\mathbf{0},t])_t$ for all $t\in q\mathring{\vee} s$, $a\in \mathcal{C}_p, b\in \mathcal{C}_s$, $\lambda=\lambda_{psqs}$ be the canonical map $p\mathring{\vee} s \rightarrow q\mathring{\vee} s$
\end{lem}
\begin{proof} By a direct calculation, we have that 
	\begin{align*}
	b\cdot f_{p,q}(a)&=f_{p\mathring{\vee} s, q\mathring{\vee} s}(b\cdot a)
	=\sum_{k\in p\mathring{\vee} s}f_{k,\lambda k}j_k(b\cdot a)\\
	&=\sum_{t\in q\mathring{\vee} s}\sum_{k\in \lambda^{-1}t}f_{k,t}j_k(b\cdot a).
	\end{align*}
Doing the projection on two sides to $\mathcal{C}_t$ term, we have $j_t(b\cdot f_{p,q}a)=\sum_{k\in \lambda^{-1}t}f_{k,t}j_k(b\cdot a)$. Then
    \begin{align*}
    \sum_{t\in q\mathring{\vee} s}x_t\otimes j_t(b\cdot f_{p,q}a)&=
    \sum_{t\in q\mathring{\vee} s}\sum_{k\in \lambda^{-1}t}x_t\otimes f_{k,t}j_k(b\cdot a)\\
    &=\sum_{k\in p\mathring{\vee} s} x_{\lambda k} \otimes f_{k,\lambda k}j_k(b\cdot a)
    \end{align*}
    as desired. The second equation follows in a similar way as above.
\end{proof}

\begin{prop}\label{Leibniz}
    Two differentials $\partial$ and $\delta$ with respect to the product (\ref{product}) of  $A^*(L,\mathcal{C})$  satisfy the  Leibniz laws
    $$\partial(\alpha\beta) = \partial(\alpha)\beta + (-1)^{i-r(p)}\alpha\partial(\beta)$$
    $$\delta(\alpha\beta)= \delta(\alpha)\beta + (-1)^{i-r(p)}\alpha\delta(\beta)$$
    for all $\alpha\in A^*([\mathbf{0},p])_p\otimes \mathcal{C}_p^i$ and $\beta\in A^*([\mathbf{0},q])_q\otimes \mathcal{C}_q^j$.
\end{prop}

\begin{proof}
	Write $\alpha=x_1\otimes c_1$ and $\beta=x_2\otimes c_2$ where $x_1\in A^*([\mathbf{0},p])_p, x_2\in A^*([\mathbf{0},q])_q, c_1 \in \mathcal{C}_p^i$, and $c_2 \in \mathcal{C}_q^j$. Let $\partial(x_1) = \sum_u x_{p_u}$ and $\partial(x_2) = \sum_v x_{q_v}$ such that $p$ covers $p_u$ and $q$ covers $q_v$. Since  the differential  of the OS-algebra $A^*(L)$ satisfies the Leibniz law,
we have that $\partial(x_1x_2)=\partial(x_1)x_2+(-1)^{r(p)}x_1\partial(x_2)=\sum x_{p_u}x_2 +(-1)^{r(p)}\sum x_1x_{q_v}$. So	
	\begin{align*}
	(-1)^{r(q)i}\partial(\alpha\beta)&=\sum_{k\in p\mathring{\vee} q}\partial ((i_kx_1 i_kx_2)\otimes j_k(c_1c_2))\\
	&=\sum_{k\in p\mathring{\vee} q}(\sum_u i_kx_{p_u}i_kx_2\otimes f_{k,\lambda k}j_s(c_1c_2)+(-1)^{r(p)}\sum_v i_kx_1i_kx_{q_v} \otimes f_{k,\lambda k}j_s(c_1c_2))
	\end{align*}
	where the first $\lambda$ is $\lambda_{pqp_uq}$, the second is $\lambda_{pqpq_v}$. Using Lemma \ref{a simple lem}, the right side equals
	\begin{align*}
	&\sum_u \sum_{t\in p_u\mathring{\vee} q}i_tx_{p_u}i_tx_2\otimes j_t(f_{p, p_u}(c_1)c_2)+(-1)^{r(p)}\sum_v \sum_{t\in p\mathring{\vee} q_v}i_tx_1i_tx_{q_v} \otimes j_t(c_1f_{q,q_v}(c_2))\\
	&= (-1)^{r(q)i}(\sum_u x_{p_u} \otimes f_{p, p_u}(c_1))(x_2\otimes c_2)
	   + (-1)^{r(p)+i(r(q)-1)}(x_1\otimes c_1)(\sum_v x_{q_v} \otimes f_{q, q_v}(c_2))\\
   &= (-1)^{r(q)i}\partial (\alpha)\beta+ (-1)^{r(p)+r(q)i-i}\alpha\partial(\beta)
	\end{align*}
	from which  our first equation follows, and $(-1)^{r(q)i}\delta(\alpha\beta)$ equals
	\begin{align*}
	&\sum_{k\in p\mathring{\vee} q}\delta ((i_kx_1 i_kx_2)\otimes j_k(c_1c_2))\\
	&= (-1)^{r(p)+r(q)}\sum_{k\in p\mathring{\vee} q}(i_kx_1i_kx_2)\otimes \delta j_k(c_1c_2)\\
	&=(-1)^{r(p)+r(q)}\sum_{k\in p\mathring{\vee} q}(i_kx_1i_kx_2)\otimes j_k(\delta(c_1)c_2) + (-1)^{r(p)+r(q)+i}\sum_{k\in p\mathring{\vee} q}(i_kx_1i_kx_2)\otimes j_k(c_1\delta(c_2))\\
	&=(-1)^{r(p)+r(q)i}[(x_1 \otimes \delta(c_1))(x_2 \otimes c_2)
	  + (-1)^{r(q)+i}(x_1 \otimes c_1)(x_2 \otimes \delta(c_2))]\\
	&=(-1)^{r(q)i}\delta(x_1 \otimes c_1)(x_2 \otimes c_2)
	   + (-1)^{r(p)+i+r(q)i}(x_1 \otimes c_1)\delta(x_2 \otimes c_2)\\
	&=(-1)^{r(q)i}\delta(\alpha)\beta
	   + (-1)^{r(p)+i+r(q)i}\alpha\delta(\beta)
	\end{align*}
	which induces  our second required equation.
\end{proof}

\begin{cor}
 Let	$\mathcal{C}$ be a monoidal presheaf of DGA on a locally geometric poset $\mathfrak{L}$, then the total complex $Tot(A^*(\mathfrak{L},\mathcal{C}))$ is a differential algebra with the total differential $\delta+\partial$.
\end{cor}
\begin{proof}
	It suffices to show that $\delta+\partial$ satisfies the Leibniz law, that is,
$$(\delta+\partial)(\alpha\beta) = (\delta+\partial)(\alpha)\beta + (-1)^{i-r(p)}\alpha(\delta+\partial)(\beta).$$
This immediately  follows from the Leibniz laws of $\delta$ and $\partial$ in Proposition~\ref{Leibniz}.
\end{proof}

Recall that the "column-wise" filtration of the double complex $A^*(\mathfrak{L},\mathcal{C})$ is $$\tau^{-k}Tot(A^*(\mathfrak{L},\mathcal{C}))=\bigoplus_{q\text{ with }r(q)\leq k} A^*(\mathfrak{L},\mathcal{C})_q.$$
If $\mathcal{C}$ is a monoidal presheaf, this is a decreasing filtration of algebra by above discussing ,  satisfying the condition  of \cite[Theorem  2.14]{McCleary2000}. Thus we have
\begin{cor}\label{SS1}
	There is a spectral sequence associated with filtration $\tau^{-*}$ satisfying
	$$E_2^{-i,j}=H^{-i}(A^*([p,\infty),H^{j}(\mathcal{C})),\partial)$$
which	converges to $H^*(Tot(A^*([p,\infty),\mathcal{C})))$ as algebras, i.e. $E_\infty^{-*,*} = Gr_*^\tau H^*(Tot(A^*([p,\infty),\mathcal{C})))$ as algebras.
\end{cor}

 Our spectral sequence with the product structure may induce some good conditions of degeneration.  
\begin{thm}\label{DG1}
	Assume that the algebra $H^{-*}(A^*([p,\infty),H^*(\mathcal{C})))$ is generated by $H^0$ and $H^{-1}$ of the chain complex $(A^*([p,\infty),H^*(\mathcal{C})),\partial)$. Then the spectral sequence in  Corollary \ref{SS1} collapse at $E_2$; namely, $$\bigoplus_{i,j}E_2^{-i,j}\cong Gr_*^\tau H^*(Tot(A^*([p,\infty),\mathcal{C})))$$ as algebras.
\end{thm}
\begin{proof}
	The "column-wise" filtration $$F^{-k}Tot(A^*([p,\infty),\mathcal{C}))=\bigoplus_{q \text{ with }r(q)\leq k+r(p)} A^*([p,\infty),\mathcal{C})_q$$ is an decreasing filtration, so the differential  $d_2$ maps $E_2^{-i,j}$ to $E_2^{-i+2,j-1}$. Then $d_2(E_2^{-1,j})=0$ and $d_2(E_2^{0,j})=0$ since $E_2^{-i,j}=H^{-i}(A^*([p,\infty),H^{j}(\mathcal{C})))=0$ for $i< 0$. Since we have assumed that $H^{-*}(A^*([p,\infty),H^*(\mathcal{C})))$ is generated by degree $H^0$ and $H^{-1}$, this means that $E_2^{i,j}$ is generated by $E_2^{1,*}$ and $E_2^{0,*}$. Furthermore,  $d_2(E_2^{-i,j})=0$ for all $i,j$ since $d_2$ satisfies the Leibniz law on $E_2$ term, so $E_2^{-i,j}=E_3^{-i,j}$, and  all $d_r$ for $r\geq 2$ are zero.
\end{proof}


\subsection{Construction of monoidal presheaf $\hat{\mathcal{C}}(\mathcal{A})$}\label{main coef sys}

Firstly, let us review some definition about presheaf of singular cochains.

\begin{defn}
	Let $C^*_0(X)$ be those singular cochains $f$ which are zero on all elements of a suitable open cover of $X$ (which may depends on $f$), define $\hat{C}^*(X)=C^*(X)/C^*_0(X)$. If $\varphi:Y\rightarrow X$ is a map of space, $\varphi^*: C^*(X)\rightarrow C^*(Y)$ maps $C^*_0(X)$ to $C^*_0(Y)$, so it induces a well defined map $\hat{\varphi}^*:\hat{C}^*(X)\rightarrow \hat{C}^*(Y)$. Let $\hat{C}^*(X,A)$ be the kernal of map $\hat{C}^*(X)\rightarrow \hat{C}^*(A)$ where $A$ is a subspace of $X$. Literally, $\hat{C}^*(X,A)$ is the subset of $\hat{C}^*(X)$, each of whose elements can be represented by a cochain $f$ which valuates  zero on all elements of a suitable open cover of $A$.
\end{defn}
\begin{rem}\label{sheaf sing}
	Let $\mathcal{S}^*:U\rightarrow C^*(U)$ be the presheaf of singular cochains on $X$ and $\hat{\mathcal{S}}^*$ be the associated sheaf.  Theorem 6.2 of \cite[Chapter I]{Bredon1997} tells us that $\hat{C}^*(X)=\varGamma(X,\hat{\mathcal{S}}^*)$ if $X$ is paracompact. $C^*_0(X)$ has zero cohomology by the discussion of subdivision, so the quotient map $C^*(X)\rightarrow \hat{C}^*(X)$ is a quasi-isomorphism of complexes.	
\end{rem}

\begin{lem}\label{lem of hat c}
	Let $A,B$ be open subspaces of a metric space $X$. Then
	\begin{enumerate}
	\item The cup product on $C^*(X)$ induces a cup product $$\cup: \hat{C}^*(X,A)\otimes \hat{C}^*(X,B)\rightarrow \hat{C}^*(X,A\cup B).$$
	\item There is a short exact sequence $$0\rightarrow \hat{C}^*(X,A\cup B)\rightarrow \hat{C}^*(X,A)\oplus  \hat{C}^*(X,B)\rightarrow \hat{C}^*(X,A\cap B)\rightarrow 0.$$
	\item If $\{A_i\}$ consists of  open subspaces of $X$ such that $A_i\cup A_j=X$ for any different $i,j$, then
	$$  \hat{C}^*(X,\bigcap_iA_i)=\bigoplus_i\hat{C}^*(X,A_i).$$
	\end{enumerate}
\end{lem}
\begin{proof}
	(1) Let $f_1$ (resp. $f_2$) be a cochain  of $\hat{C}^*(X,A)$ (resp. $\hat{C}^*(X,B)$) such that its valuation vanishes on all elements of open cover $\mathscr{U}_1$ (resp. $\mathscr{U}_2$) of $A$ (resp. $B$). Then $f_1\cup f_2$ is a cochain whose valuation vanishes on an open cover $\mathscr{U}_1\cup\mathscr{U}_2$ of $A\cup B$, representing an element of $\hat{C}^*(X,A\cup B)$. 

\vskip .1cm (2) $\hat{C}^*(X,A),\hat{C}^*(X,B)$ are both subgroups of $\hat{C}^*(X,A\cap B)$. It is clear that $\hat{C}^*(X,A)\cap \hat{C}^*(X,B)= \hat{C}^*(X,A\cup B)$. Thus it suffices to show $\hat{C}^*(X,A\cap B)=\hat{C}^*(X,A)+\hat{C}^*(X,B)$. Let $f$ represent an element of $\hat{C}^*(X,A\cap B)$. Two elements $f|A\in \hat{C}^*(A)$ and $0\in \hat{C}^*(B)$ give the same restriction $0$ on $A\cap B$, so we can "glue" them together and get an element $f_A \in \hat{C}^*(A\cup B)$ satisfying that $f_A|A=f|A$ and  $f_A|B=0$ since $U \longmapsto \hat{C}^*(U)$ is a sheaf by Remark \ref{sheaf sing}. Now let $g \in \hat{C}^*(X)$ such that $g|A\cup B=f_A$, then $f-f|A$ represents an element of $ \hat{C}^*(X,A)$ and $f|A$ represents an element of $\hat{C}^*(X,B)$. These two elements give us the required decomposition. 

\vskip .1cm (3) Notice that $A_1\cup (A_2\cap A_3\cap \cdots)=(A_1\cup A_2)\cap(A_1\cup A_3)\cap \cdots=X\cap X \cap\cdots=X$.  Then the required result follows by using (2) and an induction.
\end{proof}

Next let us return back to the manifold arrangement $\mathcal{A}$ with a quasi-intersection poset $\mathfrak{L}$.
\begin{defn}
	We define $$\hat{\mathcal{C}}(\mathcal{A})_p=\hat{C}^*(M,M-M_p)$$
\end{defn}

\begin{lem}
	$\hat{\mathcal{C}}(\mathcal{A})$ is a monoidal presheaf.
\end{lem}
\begin{proof}
	The map $f_{p,q}: \hat{\mathcal{C}}(\mathcal{A})_p \rightarrow \hat{\mathcal{C}}(\mathcal{A})_q$ induced by inclusion make $\hat{\mathcal{C}}(\mathcal{A})$ is a presheaf. The cup product $\cup$ maps $\hat{\mathcal{C}}(\mathcal{A})_p\otimes \hat{\mathcal{C}}(\mathcal{A})_q$ to $\hat{C}^*(M,M-M_p\cap M_q)$ by Lemma \ref{lem of hat c}(1), where $M_p\cap M_q=\sqcup_{ s\in p\mathring{\vee} q}M_s$. Then $\hat{C}^*(M,M-M_p\cap M_q)=\bigoplus_{ s\in p\mathring{\vee} q}\hat{\mathcal{C}}(\mathcal{A})_s$ by  Lemma \ref{lem of hat c}(3). The condition $b\cdot
	f_{p,q}(a)=f_{p\mathring{\vee} s, q\mathring{\vee} s}(b\cdot a)$ (resp. $f_{p,q}(a)\cdot b=f_{p\mathring{\vee} s, q\mathring{\vee} s}(a\cdot b)$) of monoidal presheaf is trivial in this case since each $f_{*,*}$ is an inclusion map and each side of equation is represented by $b\cup a$ (resp. $a\cup b$).
\end{proof}

\vskip .2cm

\begin{rem}\label{diff presheaf}
As two monoidal presheaves of cohomology groups, $H^*(\hat{\mathcal{C}}(\mathcal{A}))$ and $H^*(\mathcal{C}(\mathcal{A}))$ have  no any difference  by the discussion of subdivision.  They both are isomorphic to the monoidal presheaf $H^*(\mathcal{A})_p=H^*(M,M-M_p)$.
\end{rem}


	\textbf{Notation.} For any manifold arrangement $\mathcal{A}$ with a quasi-intersection poset $\mathfrak{L}$, the double complex $A^*(\mathfrak{L},\hat{\mathcal{C}}(\mathcal{A}))$ is an algebra as discussed in last subsection, abbreviate it as $A^*(\mathcal{A})$, which is called the {\em global OS-algebra} associated with  $\mathcal{A}$.

\vskip .2cm
We are going to prove a similar result of Theorem \ref{main} in the sense of algebras. For each $p\in \mathfrak{L}$, given  a similar filtration of  $\hat{\mathcal{C}}(\mathcal{A})$ as
$$ F^i\hat{\mathcal{C}}(\mathcal{A})_p=\hat{C}^*(M,F_p^iM).$$
 This is a similar "hat" version of the filtration $F^i\mathcal{C}(\mathcal{A})_p$ appeared in Definition \ref{presheaf of SS}. We can  calculate the cohomology algebra of $Tot(A^*([p,\infty),\hat{\mathcal{C}}(\mathcal{A})))$ by this filtration.
\begin{thm}\label{main alg}
	$H^*(Tot(A^*([p,\infty),\hat{\mathcal{C}}(\mathcal{A}))))$ and $H^*(N(S_p),N(S_p)_0)$ are isomorphic as algebras. In particular, for $p=\textbf{0}$, $H^*(Tot(A^*(\mathcal{A})))$ and $H^*(\mathcal{M}(\mathcal{A}))$ are isomorphic as algebras.
\end{thm}
\begin{proof}
	There is no difference between the homologies of $F^i\hat{\mathcal{C}}(\mathcal{A})_p$ and $F^i\mathcal{C}(\mathcal{A})_p$ by subdivision. So all processes in the proof of theorem \ref{main} can still be carried out very well for the filtration $F^i\hat{\mathcal{C}}(\mathcal{A})_p$. Then we first have that
	$H^*(Tot(A^*([p,\infty),\hat{\mathcal{C}}(\mathcal{A}))))\cong H^*(N(S_p),N(S_p)_0)$ as modules. This isomorphism is  induced by a quotient map of algebras, 
so it is also an isomorphism of  algebras.
\end{proof}

\subsection{Product structure on our spectral sequence}
Now  applying  Corollary \ref{SS1}, we have
\begin{cor}\label{SS alg}
	There is a spectral sequence associated with filtration $\tau^{-*}$ of the double complex $A^*([p,\infty),\hat{\mathcal{C}}(\mathcal{A}))$ such that
	$$E_1^{-i,j}=A^*([p,\infty),H^{j}(\mathcal{A}))_{-i}$$
with $d_1=\partial$, and
	 $$E_2^{-i,j}=H^{-i}(A^*([p,\infty),H^{j}(\mathcal{A})),\partial),$$
which converges to $H^*(N(S_p),N(S_p)_0)$ as algebras. In particular, for $p=\textbf{0}$, there exists a spectral sequence with  $$E_2^{-i,j}=H^{-i}(A^*(\mathfrak{L},H^{j}(\mathcal{A})),\partial),$$
which converges to $H^*(\mathcal{M}(\mathcal{A}))$ as algebras, i.e. $Gr_*^\tau H^*(\mathcal{M}(\mathcal{A}))\cong H^{-*}(A^*(\mathfrak{L},H^{*}(\mathcal{A})),\partial)$ as algebras.
\end{cor}

\begin{rem}
	Note that $H^*(\mathcal{A})$ is a monoidal presheaf on $\mathfrak{L}$, and the product structure on $E_2$  is given by the 'global' OS-algebra $A^*([p,\infty),H^*(\mathcal{C}(\mathcal{A})))$.
\end{rem}

\begin{example}
	Although the double complex in Theorem \ref{main alg} may be very complicated, the $E_1$ page of above spectral sequence can be simple in some cases. Actually we can write down it explicitly, see Theorem \ref{str of presheaf} as an example. In particular, this spectral sequence degenerates at $E_2$ page for the case of complex projective varieties, see following subsection.
\end{example}

\subsection{Arrangements of subvariety}
In this subsection, we assume that $M$ is a complex projective smooth variety and each $N_i\in \mathcal{A}$ is also a smooth subvariety.

\begin{thm}\label{ring}	
	Using $\mathbb{Q}$ as coefficients, we have $$H^*(\mathcal{M}(\mathcal{A}))  \cong H^{-*}(A^*(\mathfrak{L},H^{*}(\mathcal{A})),\partial)$$ as algebras, where the product of ring in the above right side is given by Definition \ref{OS-alg with coef}.
\end{thm}
\begin{proof}
	Deligne proved in \cite{Deligne1974} that the cohomology ring of any algebraic variety is isomorphic to the associated graded ring with respect to the weight filtration. Actually, this isomorphism can be defined by Deligne splitting
	$$I^{p,q} := F^p\cap W_{p+q}\cap(\overline{F^q}\cap W_{p+q} +\sum_{j\geq 2}\overline{F^{q-j+1}}\cap W_{p+q-j})$$
	see \cite[Lemma-Definition 3.4]{Board2008} for details.  Mapping each $\bigoplus_{p+q=k}I^{p,q}$ to $Gr_{k}^W$ by projection (notice that the cup product is with respect to $F,W$), we see that $I^{p,q}\cup I^{s,t} \subseteq I^{p+s,q+t}$, so this map preserves the product and is an natural isomorphism  $\lambda: H^*(U)\rightarrow \bigoplus_kGr_k^WH^*(U)$ as algebras.

\vskip .2cm	
	Compare Corollaries \ref{alg arg} and \ref{SS alg}, the weight filtration and $\tau^{-*}$ have the same $E_2$ page, so the spectral sequence in Corollary \ref{SS alg} also degenerates on $E_2$, which means that $Gr_*^{\tau}H^*(\mathcal{M}(\mathcal{A}))\cong E_2^{*,*}$ as algebras. Meanwhile, $Gr_*^{\tau}H^*(\mathcal{M}(\mathcal{A}))=Gr_*^{W}H^*(\mathcal{M}(\mathcal{A}))$ by Theorem \ref{compare thm}, so $H^{-*}(A^*(\mathfrak{L},H^{*}(\mathcal{A}))\cong E_2^{*,*}\cong Gr_*^{W}H^*(\mathcal{M}(\mathcal{A}))\cong H^*(\mathcal{M}(\mathcal{A}))$ as algebras (probably with different mixed Hodge structures).
\end{proof}

\section{Application to chromatic configuration spaces}\label{chromatic space}
\subsection{Chromatic configuration space}
In the classical vertex coloring problem of a graph $G$, a usual way we use is to color the vertices of $G$ with  $m$ colors from $[m]=\{1, ..., m\}$, so that adjacent vertices would receive different colors, a so-called proper $m$-coloring. It is well-known that the number $\chi_G(m)$ of proper $m$-colorings  is a polynomial of $m$, called the {\em chromatic polynomial} of $G$. In this section, we consider the problem, but we will use  a manifold $M$ as a color set to color the vertices of $G$, and the resulting colorings will also form a manifold,  called the chromatic configuration space of $M$ on $G$.

\begin{defn}\label{coloring}
Let $G$ be a simple graph with  vertex set $[n]=\{1, ..., n\}$ and  $M$ be a smooth manifold without boundary.	Then  the \textbf{chromatic configuration space} of $M$ on $G$ consists of all the proper colorings of $G$ with all points of $M$ as colors
	$$F(M,G) = \{(x_1,...,x_n)\in M^n|(i,j)\in E(G)\Rightarrow x_i\neq x_j\}$$
	where $E(G)$ denotes the set of all edges of $G$.
\end{defn}

  This generalizes the concept of the classical configuration space  $$F(M,n)=\{(x_1,...,x_n) \in M^n |i\neq j\Rightarrow x_i\neq x_j\}.$$ Actually, $F(M,G)=F(M, n)$ when $G$ is a complete graph. In the viewpoint of configuration spaces, the definition of $F(M, G)$ first appeared in the work of  Eastwood and Huggett~\cite{Eastwood2007}, where $F(M, G)$ was called the generalized configuration space, and the case in which $M$ is a  Riemann surface  was studied in \cite{Berceanu2017}. In addition,  Dupont in~\cite{Dupont2013} also studied the hypersurface arrangements. In this section, we will study the more general case of $F(M,G)$ by our approach about manifold arrangements.

\vskip .2cm
Following the assumption and the notion in Definition~\ref{coloring}, by $L_G$ we denote the associated geometric lattice of $G$, also see Example \ref{ex graph}. It is well-known that $L_G$ is a geometric lattice (the matroid associated with this lattice is also known as the cycle matroid of $G$), the partition induced by a single edge is an atom of $L_G$. By $\mathbf{0}$ and $\mathbf{1}$ we also denote the minimum and maximum element of $L_G$, respectively.

\begin{defn}
	For $x=(x_1, ..., x_n)\in  M^n$,  
define $G(x)$ to be the spanning subgraph given by those edges $(i,j)$ in $E(G)$  with $x_i=x_j$. By $L_G(x)$ we denote the element of $L_G$ induced by the subgraph $G(x)$.
For $p\in L_G$, define the $L_G$-indexed 'diagonal' of $M^n$ as $\Delta_p=\{x\in M^n| L_G(x)\geq p\}$.
\end{defn}

\begin{thm}
	 The set of diagonal  $\mathcal{A}_G=\{\Delta_{a}|a\in \text{\rm Atom}(L_G)\}$ is a manifold arrangement in $M^n$ such that
 each $\Delta_{a}$ is closed in $M^n$, and  the  intersection lattice of $\mathcal{A}_G$ is just $L_G$. Choose quasi-intersection poset to be intersection lattice. In particular,
	 for $p\in L_G$,  $(M^n)_p$ equals the diagonal $\Delta_p$ defined as above.
	
\end{thm}

\begin{proof} In order to show that $\mathcal{A}_G=\{\Delta_{a}|a\in \text{\rm Atom}(L_G)\}$ is a manifold arrangement, it suffices to check that $\mathcal{A}_G$ is locally diffeomorphic to a subspace arrangement. In fact, choose $x=(x_1,x_2,...,x_n)\in M^n$ and its sufficient small neighborhood $U=U_1\times U_2\times\cdots\times U_n$ with $x_i\neq x_j\Rightarrow U_i\cap U_j= \emptyset$ and $x_i=x_j \Rightarrow U_i=U_j$, such that each $U_i$ is diffeomorphic to some $\mathbb{R}$-linear space $W_i$.  Clearly, the intersection $\Delta_{a}\cap U$ is some diagonal subspace (i.e., the subspace with some coordinates being equal) of $\prod U_i$, which is diffeomorphic to some diagonal subspace of $W=\prod W_i$. Since each diagonal subspace of $ W$ is linear,  the diffeomorphism  between $U$ and $W$ maps $\{\Delta_{a}\cap U\}$ onto a subspace arrangement in $W$.

\vskip .2cm
	It is obvious that each diagonal $\Delta_{a}$ is closed. Since
	$\Delta_p \cap \Delta_q= \{x\in M^n| L_G(x)\geq p \text{~and~} L_G(x)\geq q\} = \{x\in M^n| L_G(x)\geq p\vee q\} = \Delta_{p\vee q}$, we see that the intersection lattice of arrangements $\Delta_{a}, a\in \text{\rm Atom}(L_G)$ is just $L_G$ and $(M^n)_p=\Delta_p$.
\end{proof}


Consider the monoidal presheaf $\hat{\mathcal{C}}(\mathcal{A}_G)$ and the model $A^*(\mathcal{A}_G)=A^*(L_G,\hat{\mathcal{C}}(\mathcal{A}_G))$ associated with arrangement $\mathcal{A}_G$.

\begin{thm}\label{SS of chromatic}
	$H^*(F(M,G))$ is isomorphic to  $H^*(Tot(A^*(\mathcal{A}_G)))$ as algebras. In particular, there is a spectral sequence with  $$E_1^{-i,j}=A^*(L_G,H^{j}(\mathcal{A}_G))_{-i} \text{ with } d_1=\partial, $$
	which converges to $H^*(F(M,G))$ as algebras.
\end{thm}

\begin{proof}
	Applying theorem \ref{main alg}.
\end{proof}

Using a field as coefficients, the dimension of $E_1^{-i,j}$ term would be easy to calculate as follows: $$\dim E_1^{-i,j}=\sum_{p\text{ with }r(p)=i}\dim A^*(L_G)_p \dim H^{j-r(p)m}(\Delta_p)$$ by definition.  This formula will be more elegant if we consider a  polynomial with two variables
 $$P_{M,G}(s,t)=\sum_{i,j}\dim E_1^{-i,j}s^{-i}t^j.$$

\begin{lem}\label{2var poly}
	Let $M$ be a $m$-dimensional manifold without boundary, and $G$ be a simple graph with $n$ vertexes. Then
	$$P_{M,G}(s,t)=(-1)^ns^{-n}t^{mn}\chi_G(-P(M,t)st^{-m})$$ where $\chi_G$ is the chromatic polynomial of $G$, and  $P(M,t)=\sum_i \dim H^i(M)t^i$ is the Poincar\'{e} polynomial of $M$.
\end{lem}
\begin{proof}
	 It is well-known that $\dim(A^*(L_G)_p)=(-1)^{r(p)}\mu(0,p)$ where $\mu$ is the M\"{o}bius function on $L_G$ and notice that $\dim H^j(M^n,M^n-\Delta_p)= \dim H^{j-r(p)m}(\Delta_p)$ by Thom isomorphism. Then we have
	 \begin{align*}
	 P_{M,G}(s,t)&=\sum_{i,j}\dim E_1^{-i,j}s^{-i}t^j\\
	 &=\sum_{p\in L_G,j}\dim A^*(L_G)_p\dim H^{j-r(p)m}(\Delta_p)s^{-r(p)}t^j\\
	 &=\sum_{p\in L_G}(-1)^{r(p)}\mu(0,p)s^{-r(p)}P(M,t)^{n-r(p)}t^{r(p)m}.
	 \end{align*}
	  Moreover, the  required equation follows from  the following known result for the chromatic polynomial  $$\chi_G(t)=\sum_{p\in L_G} \mu(0,p)t^{n-r(p)}.$$
\end{proof}

\begin{rem}
	The "Deletion--contraction" formula $\chi_G=\chi_{G-e}-\chi_{G/e}$ of chromatic polynomial induces a "Deletion--contraction" formula of $P_{M,G}$. It is easy to check that $P_{M,G}=P_{M,G-e}+s^{-1}t^mP_{M,G/e}$.
\end{rem}
\begin{rem}
	We can also consider the 2-variable polynomial of $E_\infty$ term for every manifold arrangements $\mathcal{A}$, i.e., $P_{\infty,\mathcal{A}}(s,t)=\sum_{i,j}\dim(E_\infty^{-i,j})s^{-i}t^j$, which is hard to calculate in general. This polynomial is a natural invariant of $\mathcal{A}$ that contains much more information than the Poincar\'{e} polynomial of $\mathcal{M}(\mathcal{A})$.
\end{rem}

\subsection{More explicit result about $H^*(F(M,G))$}\label{explicit result}
For some special case, the spectral sequence in Section 4 is simpler.
\vskip .2cm
The first case is that $M$ be a complex projective smooth variety. Combine Corollary \ref{alg arg}, Theorem \ref{ring} and last subsection, we have
\begin{cor}\label{explicit chro alg}
	Assume $M$ is a complex projective smooth variety and we use $\mathbb{Q}$ as coefficients. Then the mix Hodge structure of $H^*(F(M,G))$ satisfies
	$$Gr_{k+i}^WH^k(F(M,G)) \cong H^{-i}(A^*(L_G,H^{k+i}(\mathcal{A}_G)),\partial)$$
	and there is the following isomorphism
	$$H^*(F(M,G)) \cong H^{-*}(A^*(L_G,H^{*}(\mathcal{A}_G)),\partial)$$ as algebras.
\end{cor}

\vskip .2cm
For a general smooth manifold $M$, the spectral sequence in Theorem \ref{SS of chromatic} may have non-trivial higher differential. This spectral sequence will be simple if the diagonal cohomology class of $M^2$ is zero.
Recall (see Milnor's book \cite{Milnor1974}) that the diagonal cohomology class of $M^2$ is the image of the Thom class under the map $H^*(M^2,M^2-\Delta(M^2))\rightarrow H^*(M^2)$, where $\Delta(M^2)$ is the diagonal of $M^2$.  From now on, we always use
$\mathbb{Z}_2$-coefficients, and assume that the diagonal cohomology class of $M^2$ vanishes. This is  the case  if $M=M'\times \mathbb{R}$ where $M'$ is another manifold.

\vskip .2cm
Under the above assumption, the monoidal presheaf $H^*(\mathcal{A}_G)$ becomes simpler, and it is completely only determined by $G$ and $H^*(M)$. Firstly, observe that $H^*(\mathcal{A}_G)_p=H^*(M^n,M^n-\Delta_p)=H^*(\Delta_p)[-mr(p)]$ (let $\dim(M)=m$) by Thom isomorphism, so every element of  $H^i(\mathcal{A}_G)_p$ can be represented by an element $x_p \in H^{i-mr(p)}(\Delta_p)$.

\begin{thm}\label{str of presheaf}
	 Assume that the diagonal cohomology class of $M^2$ is zero and we use
	 $\mathbb{Z}_2$ as coefficients and $G$ is a simple graph with  vertex set $[n]$. Let $E^*(M,G)$ be the exterior algebra over ring $H^*(M)^{\otimes n}$, generated by elements $e_{ij}$ where $(ij)$ is edge of $G$. The $E_1$ page $A^*(L_G,H^*(\mathcal{A}_G))$ has zero differential ($\partial=0$) and is isomorphic to $E^*(M,G)/I$ as algebras where $I$ is the ideal generated by elements  \begin{align}
	 (1^{i-1}\otimes x \otimes 1^{n-i-1}-1^{j-1}\otimes x \otimes 1^{n-j-1})e_{ij} \tag{1}\\
	 \sum_s e_{i_1i_2}\ldots \widehat{e_{i_si_{s+1}}}\ldots e_{i_ki_1} \tag{2}
	 \end{align}
	
for all edges $(i,j)$ and all cycles $(i_1i_2,i_2i_3,\ldots,i_ki_1)$.
\end{thm}

The proof will be completed in appendix.

\begin{rem}
	These elements in (1) of Theorem~\ref{str of presheaf} also appeared in Cohen and Taylor's model when $G$ is complete graph, see \cite[page 111]{Cohen1978}. Elements in (2) of Theorem~\ref{str of presheaf} appeared in the definition of OS-algebra.
Of course, there also exists a similar result for the structure of the monoidal  presheaf $H^*(\mathcal{A}_G)$ with $\mathbb{Z}$ coefficients if we consider the orientation of the Thom class carefully. Here we use $\mathbb{Z}_2$ as coefficients only for a simplicity since it is  enough to illustrate the application of our approach in this section.
\end{rem}

\begin{cor}\label{explicit}
	Assume that the diagonal cohomology class of $M^2$ is zero and  using
	$\mathbb{Z}_2$ coefficients. Let $G$ be a simple graph with  vertex set $[n]$. Then  there exists a filtration of $H^*(F(M,G))$ such that $Gr(H^*(F(M,G)))$ is isomorphic to $A^*(L_G,H^*(\mathcal{A}_G))$ as algebras.
\end{cor}

\begin{proof}
	The differential $\partial$ of the global OS-algebra $A^*(L_G,H^*(\mathcal{A}_G))$ is zero by Theorem \ref{str of presheaf}, so
	$$E_2^{-i,j}=E_1^{-i,j}=A^*(L_G,H^{j}(\mathcal{A}_G))_{-i}.$$
	Now we only need to show that this algebra is generated by elements of $E_2^{-1,j}$ and $E_2^{0,j}$,  by making use of the degeneration condition in Theorem \ref{DG1}. This is obvious by Theorem \ref{str of presheaf}.
\end{proof}

\begin{example}
	Assume that the diagonal cohomology class of $M^2$ is zero and using
	$\mathbb{Z}_2$ coefficients. Let $G$ be the cycle graph $C_n$ of length $n$, then $Gr(H^*(F(M,G)))$ is isomorphic to the exterior algebra over ring $H^*(M)^{\otimes n}$, generated by $e_1,e_2,\ldots,e_n$, with relations
	$$(1^{i-1} \otimes x \otimes 1^{n-i} )e_i=(1^{i} \otimes x \otimes 1^{n-i-1}) e_i \text{ for } i<n$$ $$(1^{n-1} \otimes x)e_n=( x \otimes 1^{n-1}) e_n$$ $$\sum_s e_{1}\ldots \widehat{e_{s}}\ldots e_{n}=0.$$
\end{example}

Consider the Poincar\'{e} polynomial, we have
\begin{cor} With the same assumption as in Corollary  \ref{explicit}.  Then
	$$P(F(M,G))=(-1)^nt^{n(m-1)}\chi_G(-P(M)t^{1-m})$$
 where $P(-)$ denote the Poincar\'{e} polynomial of $\mathbb{Z}_2$-cohomology with a variable $t$,  $\dim M=m$, and $\chi_G(t)$ is the chromatic polynomial of $G$.
\end{cor}
\begin{proof}
	In the proof of Corollary \ref{explicit}, the spectral sequence collapses on $E_1$ term, so the Poincar\'{e} polynomial of $F(M,G)$ equals $P_{M,G}(t,t)$, where $P_{M,G}$ is the two-variable polynomial of $E_1$ term which is defined in Lemma \ref{2var poly}. The required equation is a direct result of Lemma \ref{2var poly}.
\end{proof}
\begin{rem}
	With the same assumption of the above Corollary, $P(F(M,G))$ satisfies a similar "Deletion--contraction" formula  $$P(F(M,G))=P(F(M,G-e))+t^{m-1}P(F(M,G/e))$$ by the "Deletion--contraction" formula of chromatic polynomial.
\end{rem}

In a more special condition, we can overcome the gap between $Gr(H^*(F(M,G)))$ and $H^*(F(M,G))$ by a simple method of "counting degree".

\begin{thm}\label{alg}
	Using
	$\mathbb{Z}_2$ coefficients, assume $\dim M=m$ and the diagonal cohomology class of $M^2$ is zero and $H^i(M)=0$ for all $i\geq (m-1)/2$. Then $H^*(F(M,G))$ is isomorphic to $A^*(L_G,H^*(\mathcal{A}_G))$ as algebras.
\end{thm}

\begin{proof}
	Firstly, observe that if $G$ is a disjoint union of some connected graphs $G_i$, then $F(M,G)=\prod F(M,G_i)$ and $A^*(\mathcal{A}_G)=\bigoplus A^*(\mathcal{A}_{G_i})$, so it suffices to prove the case that $G$ is connected.

\vskip .2cm
	
	We perform our work by induction on the number of edges. If $G$ has no edge, then the result is obvious. Now we assume that  the result is true for all connected $G$ that $|E(G)|<s$. Consider the case in which $G$ is connected with $|E(G)|=s$.

\vskip .2cm	
	Corollary \ref{explicit} says that there is an isomorphism $\eta: A^*(L_G,H^*(\mathcal{A}_G))\rightarrow H^*(F(M,G))$ as modules, so we need to show this is also an isomorphism as algebras.

\vskip .2cm	
	Given two elements $x_p \in A^*(L_G,H^*(\mathcal{A}_G))_p$ and $x_q \in A^*(L_G,H^*(\mathcal{A}_G))_q$, where $p,q\in L_G$ with $r(p)=i$ and $r(q)=j$.
	If $p\vee q<\mathbf{1}$,  since we have assumed $G$ is connected, this means that $p\vee q$ must have two or more components, so the subgraph~\footnote{For definition of subgraph $G|p\vee q$, see Example \ref{ex graph}} $G|p\vee q$ must have less edges than $G$. Consider the arrangement $\mathcal{A}_{G|p\vee q}$ associated with the space $F(M,G|p\vee q)$, we know that this arrangement is a sub-arrangement of $\mathcal{A}_G$ as discussed in Section \ref{inc of sub}. Then Theorem \ref{sub arrangement} and Corollary \ref{explicit} give us a commutative diagram	
	\[
	\begin{CD}
	H^*(F(M,G|p\vee q)) @>i^*>> H^*(F(M,G))\\
	@A\theta AA @A\eta AA\\
	A^*([0,p\vee q],H^*(\mathcal{A}_{G|p\vee q})) @>j^*>> A^*(L_G, H^*(\mathcal{A}_G))
	\end{CD}
	\]
where $\eta,\theta$ are two module-isomorphisms (note that $\theta$ is also an algebra-isomorphism by induction hypothesis),  $j^*$ is a natural inclusion of algebra, and $i^*$ is induced by the inclusion $F(M,G)\hookrightarrow F(M,G|p\vee q)$.
It is easy to see that those two elements $x_p$, $x_q$ and their product are located in sub-algebra $A^*([0,p\vee q], H^*(\mathcal{A}_{G|p\vee q}))$ since $\theta, i^*$ preserve product, so $\eta$ is also compatible with the    product of those two elements by the diagram.
	
\vskip .2cm	
	
	Now assume that $p\vee q=\mathbf{1}$. We know that $\eta$ induces an algebra-isomorphism: $A^*(L_G,H^*(\mathcal{A}_G))\rightarrow Gr(H^*(F(M,G)))$. Then we can write the image of the product as $$\eta(x_px_q)= \eta(x_p)\cup \eta(x_q)+\sum_i \eta (x_{p_i})$$ for some $x_{p_i} \in A^*(L_G,H^*(\mathcal{A}_G))_{p_i}, p_i<\mathbf{1}$.
	We see that $r(p)+r(q)\geq r(\mathbf{1})=n-1$ since $G$ is connected, $\eta(x_px_q)$ should have degree at least $(n-1)(m-1)$ in $H^*(F(M,G))$.

\vskip .2cm
	
	Let us observe the degree of elements $x_{p_i}$ with  $p_i<\mathbf{1}$. Write $k=r(p_i)$ and $x_{p_i}=a_{p_i}\otimes c_{p_i}$ where $a_{p_i}\in A^*(L_G)_{p_i}$ and  $c_{p_i}\in H^*(\mathcal{A}_G)_{p_i}=H^*(M^n,M^n-\Delta_{p_i})$. Then $a_{p_i}$ has degree $-k$, and  $c_{p_i}$ equals the product of a Thom class $u_{p_i}$ with $\deg u_{p_i}=mk$ and  some nonzero elements of  degree less than $(n-k)(m-1)/2$ by the assumption of this theorem. Thus, $\deg(a_{p_i}\otimes c_{p_i})< (m-1)k+(n-k)(m-1)/2=(m-1)(n+k)/2$, which implies that $k\leq n-2$, so $\deg x_{p_i}<(m-1)(n-1)$, a contradiction of $\deg\eta(x_px_q)\geq (n-1)(m-1)$. Then this forces these items $x_{p_i}$ to be zero. Therefore, $\eta(x_px_q)= \eta(x_p)\cup \eta(x_q)$.
\end{proof}

\begin{rem}
	If $M=\mathbb{R}^m$, then Theorem~\ref{alg} agrees with the classical result of $H^*(\mathbb{R}^m,n)$ given by F. Cohen.
In addition, Theorem~\ref{alg} also shows a "standard" process how to apply our main result: (i) Determine the structure of monoidal presheaf $H^*(\mathcal{A})$. (ii) Calculate the $E_2$-term  $H^{-j}(A^*(L,H^i(\mathcal{A})),\partial)$ and check the condition of degeneration (iii) Observe the gap between $H^*(\mathcal{M}(\mathcal{A}))$ and $Gr(H^*(\mathcal{M}(\mathcal{A})))$.
	Using this process, we can also easily reprove the classical result of Orlik-Solomon for complex hyperplane arrangements.
\end{rem}


\section{Appendix}
In this section, we review some known result of spectral sequence of filtrated differential algebra, see \cite{McCleary2000}.
\subsection{Spectral sequence of filtrated differential algebra}
Let $A=\bigoplus_p A^p$ be a graded module, and  $...\subset F^pA^{p+q}\subset F^{p-1}A^{p+q}\subset...$ be an decreasing filtration with the differential $d$ such that $d(F^pA^{p+q})\subset F^pA^{p+q+1}$. Define
$$Z_r^{p,q}= F^pA^{p+q}\cap d^{-1}(F^{p+r}A^{p+q})$$
$$B_r^{p,q}= F^pA^{p+q}\cap d(F^{p-r}A^{p+q-1})$$
$$Z_{\infty}^{p,q} = F^pA^{p+q}\cap \ker(d)$$
$$B_{\infty}^{p,q} = F^pA^{p+q}\cap \text{im}(d)$$
$$E_r^{p,q}=Z_r^{p,q}/(Z_{r-1}^{p+1,q-1}+B_{r-1}^{p,q})$$

\begin{prop}
	The differential $d$ which maps $Z_r^{p,q}$ to $Z_r^{p+r,q-r+1}$  induces the differential $d_r: E_r^{p,q}\rightarrow E_r^{p+r,q-r+1}$ of the associated spectral sequence, such that
	$$H^*(E_r^{*,*},d_r)=E_{r+1}^{*,*}$$
	$$E_1^{p,q}=H^{p+q}(F^pA/F^{p+1}A)$$
	$$E_{\infty}^{p,q}= F^pH^{p+q}(A,d)/F^{p+1}H^{p+q}(A,d)$$
	where $F^pH^*(A,d)=\text{Im}(H^*(F^pA,d)\rightarrow H^*(A,d))$.
\end{prop}

Furthermore, if $A$ is also an algebra, then we have
\begin{prop}\label{ss of alg}
	Suppose that $(A,d,F^*A)$ is a decreasing filtered differential graded algebra with product $ A\otimes A \rightarrow A$ satisfying $$ F^pA\cdot F^qA\subset F^{p+q}A$$
	Then there is an induced product on $E_r^{*,*}$ satisfying $$ E_r^{p,q}\cdot E_r^{s,t} \subset E_r^{p+s,q+t}$$ and $$d_r(x\cdot y)=d_r(x)\cdot y + (-1)^{p+q}x\cdot d_r(y)$$
	where $x\in E_r^{p,q}$ and $y \in E_r^{s,t}$.
	If the filtration $F^*A$ is bounded, then the spectral sequence $(E_r^{*,*},d_r)$ converges to $H(A,d)$ as algebras, i.e.,  $\bigoplus_{p,q}E_{\infty}^{p,q}$ is isomorphic to the associated graded algebra $Gr(H(A,d))=\bigoplus_{p,q}F^pH^{p+q}(A,d)/F^{p+1}H^{p+q}(A,d)$ and this isomorphism obeys the bigrading $(p,q)$.
\end{prop}

\begin{rem}
	$Gr(H^*(A,d))$ determines $H^*(A,d)$ up to the extension problem. If we use a field as coefficients or $E_\infty$ is free with $\mathbb{Z}$ coefficients, the extension problem is trivial, i.e.,  $H^*(A,d) \cong Gr(H^*(A,d))$ as modules. This isomorphism will also obey the product of $A$ in some special case. For example, if there exists an integer $k$ such that $x \in F^lH^*(A,d) \Leftrightarrow \deg x \leq kl$ for any homogeneous element $x$, then $H^*(A,d) \cong Gr(H^*(A,d))$ as algebras by dimensional reason.
\end{rem}

\subsection{Proof of Theorem \ref{str of presheaf}}
In this subsection, we give the proof of  Theorem \ref{str of presheaf}.
First we prove a lemma.
\begin{lem}\label{fundamental}
	Assume that $V_1, V_2$ are two linear subspaces of $\mathbb{R}^n$ such that  $V_1\cap V_2=0$, $\dim(V_1)=i$ and $\dim(V_2)=n-i$. By $\mu_1,\mu_2, \mu_3$ we denote the fundamental classes of $H^{n-i}(\mathbb{R}^n,\mathbb{R}^n-V_1), H^{i}(\mathbb{R}^n,\mathbb{R}^n-V_2), H^{n}(\mathbb{R}^n,\mathbb{R}^n-0)$, respectively. Then $ \mu_1\cup\mu_2=\mu_3$.
\end{lem}
\begin{proof}
	Write $\mathbb{R}^n=\mathbb{R}_1 \times \mathbb{R}_2\times \cdots\times \mathbb{R}_n$ be the product of $n$-copies of $\mathbb{R}$. Without the loss of generality, assume that $V_1=\mathbb{R}_1 \times \cdots\times\mathbb{R}_i\times 0$ and $V_2=0\times\mathbb{R}_{i+1} \times\cdots\times \mathbb{R}_n$. Let $e^i$ be the fundamental class of $H^1(\mathbb{R}_i,\mathbb{R}_i-0)$. Then we can write $\mu_1=1\times\cdots\times 1\times e^{i+1}\times\cdots\times e^{n}$, $\mu_2= e^1\times\cdots\times e^{i}\times 1\times \cdots \times 1$, so  $\mu_1 \cup \mu_2 = e^1\times e^{2}\times \cdots\times e^{n}$, which is just the fundamental class of $H^{n}(\mathbb{R}^n,\mathbb{R}^n-0)$.
\end{proof}

\begin{proof}[Proof of Theorem \ref{str of presheaf}]
	Let $u_p$ be the Thom class of $H^*(M,M-\Delta_p)$ for any $p\in L_G$. Assume that $p,q$ are independent. Then $r(p)+r(q)=r(p\vee q)$, which means that $\dim \Delta_p+\dim \Delta_q=\dim M^n+\dim \Delta_{p\vee q}$. Let $F_x$ be the fiber at $x$ of the normal bundle $N(\Delta_{p\vee q})$ (here we do not distinguish the normal bundle and the tubular neighborhood).  We can choose suitable local charts and tubular neighborhoods such that $F_x\cap \Delta_p$ and $F_x\cap \Delta_p$ are linear subspaces of $F_x$, denoted by $V_1,V_2$, respectively. Then we know that $V_1\cap V_2=0$ and $V_1\oplus V_2=F_x$ since $p,q$ are independent. Moreover,  $N(\Delta_p)_0 \cap F_x = F_x-V_1$ and $N(\Delta_q)_0 \cap F_x = F_x-V_2$, so $u_p|_{(F_x,F_x-V_1)}$ is the fundamental class of $H^{r(p)m}(F_x,F_x-V_1)$ and $u_q|_{(F_x,F_x-V_2)}$ is the fundamental class of $H^{r(q)m}(F_x,F_x-V_1)$. By Lemma~\ref{fundamental},  $(u_p \cup u_q)|_{(F_x,F_x-0)}$ is the fundamental class of $H^{r(p\vee q)m}(F_x,F_x-0)$  for every $x \in \Delta{p\vee q}$. Thus,  $u_p \cup u_q$ must be the Thom class $u_{p\vee q}$ of $H^*(M^n,M^n-\Delta_{p\vee q})$ by uniqueness of Thom class.

\vskip .2cm
	
	Now we want to check that $u_p\cup u_q=0$ if $p,q$ are dependent.  Firstly, assume that $p$ is an atom. Then we know that $(M^n,M^n-\Delta_p)=(M^2,M^2-\Delta(M^2))\times M^{n-2}$, so $u_p$ is just the product of the Thom class $u$ of $(M^2,M^2-\Delta(M^2))$ with the unit of $H^*(M^{n-2})$, and $u^2= u\cdot w_m(T(M))$ by Milnor's book \cite{Milnor1974}, where $w_m(T(M))$ is the top Stiefel-Whitney class of the tangent bundle $T(M)$. Furthermore, $u^2=0$ since $w_m(T(M))$ is the image of the diagonal cohomology class, so $u_p^2$ is also zero for the atom $p$. Thus, for every $p$,  $u_p^2=0$  since $p$ is always a join of some independent atoms, and so $u_p \cup u_q=0$ for $p\leq q$ since $u_q$ can be expressed as $u_p\cup u_{s_1}\cup u_{s_2}\cdots$ for some independent atoms $s_i$. For any  dependent $p,q$, let $q=s_1\vee s_2\cdots $ for some independent atoms $s_i$, then there exists some $k$ such that $p\vee s_1\vee\cdots \vee s_k\geq s_{k+1}$, so $u_p \cup u_q=u_{p\vee s_1 \vee\cdots \vee s_k}\cup u_{s_{k+1}}\cup\cdots$, which must be equal to zero by the above discussion.

\vskip .2cm
    Now, Let $\psi_{p,q}:H^*(\Delta_{p})\rightarrow H^*(\Delta_q)$ be the map induced by inclusion for $p\leq q$.
	Notice that every $H^*(\mathcal{C}(\mathcal{A})_p)=H^*(M^n,M^n-\Delta_p)$ is an $H^*(M^n)$-algebra such that $u_p\cup x = u_p \cdot \psi_{0,p}(x)$ for $x\in H^*(M^n)$, it is not difficult to check that
	$(u_p\cdot x_p)\cup (u_q\cdot x_q)=u_{p\vee q}\cdot (\psi_{p,p\vee q}(x_p)\cup \psi_{q,p\vee q}(x_q))$. Then we can define $E^*(M,G)/I\rightarrow A^*(L_G,H^*(\mathcal{A}_G))$ by map $e_{ij}\cdot x$ to $e_{ij}\otimes\psi_{0,p}(x)$ (recall that OS-algebra $A^*(L_G)$ is the exterior algebra generated by $e_{ij}$ with relations $\sum_s e_{i_1i_2}\ldots \widehat{e_{i_si_{s+1}}}\ldots e_{i_ki_1}=0$ for all cycle), it's a well defined injective map of algebra by above discussion, it's an isomorphism by checking the dimension of each side.
	
	The complex $(A^*(L_G,H^*(\mathcal{A}_G)),\partial)$ have zero differential by checking that $f^*_{p,q}=0$ for all $p\geq q$. We only need to check the case that $p$ covers $q$ by the functorial property of $f^*$. Assume that $p=q\vee s$ for an atom $s$. Then  we have that $f^*_{p,q}(u_p)=f^*_{p,q}(u_q \cup u_s)=u_q \cup f^*_{s,0}(u_s)$, which is zero since $f^*_{s,0}(u_s)$ is the product of the diagonal cohomology class with some unit.
\end{proof}

\end{document}